\documentclass[english,11pt, oneside, a4paper]{amsart}

\usepackage{biblatex}
\usepackage{amsfonts,amsmath,amsthm,amssymb,amscd,latexsym,enumerate,etoolbox,mathrsfs,comment, graphicx}
\usepackage[left=2.5cm,right=2.5cm,top=2cm,bottom=2.5cm,bindingoffset=5mm]{geometry}
\usepackage[T1]{fontenc} 
\usepackage[utf8]{inputenc}
\usepackage{mathtools}
\usepackage{hyperref}
\usepackage{theoremref}
\usepackage{quiver}
\usepackage{enumitem} 

\usepackage{biblatex} 
\newtheorem{theorem}{Theorem}[section]
\newtheorem{proposition}[theorem]{Proposition}
\newtheorem{lemma}[theorem]{Lemma}
\newtheorem{corollary}[theorem]{Corollary}
\theoremstyle{definition}
\newtheorem{remark}[theorem]{Remark}
\newtheorem{definition}[theorem]{Definition}
\newtheorem{example}[theorem]{Example}

\newcommand{\mb}[1]{\mathbb{#1}}

\newcommand{\norm}[1]{\left\| #1  \right\|}

\newcommand{\Z}{\mb Z}
\newcommand{\N}{\mb N}
\newcommand{\C}{\mb C}
\newcommand{\V}{\mathcal{V}}

\newcommand{\E}{\mathcal{E}}
\newcommand{\CP}{\mathcal{O}}
\newcommand{\calL}{\mathcal{L}}

\newcommand{\K}{\mathcal{K}}

\DeclareMathOperator{\im}{Im}

\DeclareMathOperator{\supp}{supp}
\DeclareMathOperator{\tr}{tr}

\DeclareMathOperator{\End}{End}

\DeclareMathOperator{\orb}{Orb}

\DeclareMathOperator{\dimnuc}{dim_{nuc}}

\DeclareMathOperator{\dr}{dr}

\DeclareMathOperator{\fin}{f}
\DeclareMathOperator{\gl}{gl}

\title[Cuntz--Pimsner algebras of partial automorphisms twisted by vector bundles II]{Cuntz--Pimsner algebras of partial automorphisms twisted by vector bundles II: Nuclear dimension}
\author{Aaron Kettner}

\address{Department of Abstract Analysis\\ Institute of Mathematics, Czech Academy of Sciences, \v{Z}itn\'a 25, 115 67 Prague 1, Czech Republic}
\email{kettner@math.cas.cz}

\date{\today}
\subjclass[2020]{37A55, 46L35, 46L08}

\thanks{Funded by GA\v{C}R project GF22-07833K and \mbox{RVO: 67985840}. Part of this work was carried out while funded by GA\v{C}R project 20-17488Y. This work has been supported by Charles University Research Centre
program No. UNCE/24/SCI/022, and by the Charles University project SVV-2023-260721. The author is currently funded by GA\v{C}R project G25-15403K.}

\keywords{Partial automorphisms, $\mathrm{C}^*$-correspondences, classification of nuclear \mbox{$\mathrm{C}^{*}$-algebras}}

\begin{document}
\maketitle
\begin{abstract}
We show that Cuntz--Pimsner algebras associated to partial automorphisms twisted by vector bundles are classifiable in the sense of the Elliott program whenever the action is minimal and the base space is compact, infinite and has finite covering dimension. We also investigate the tracial state space of our algebras, and show that traces are in bijection to certain conformal measures. This generalizes results about partial crossed products by Geffen and complements results about the $C^*$-algebras associated to homeomorphisms twisted by vector bundles of Adamo, Archey, Forough, Georgescu, Jeong, Strung and Viola. We use our findings to generalize various existing statements about orbit-breaking subalgebras. 
\end{abstract}

\tableofcontents

\section{Introduction}

The classification program for $C^*$-algebras was initiated by George Elliott in the 1990's. It has had tremendous success in recent years, cumulating in the theorem that $C^*$-algebras which are unital, separable, simple, infinite-dimensional, have finite nuclear dimension and satisfy the UCT, are completely classified by the so-called Elliott invariant. This is an invariant consisting of $K$-theoretical as well as tracial data, see for example \cite{CETWW, BBSTWW:2Col, EllGonLinNiu:ClaFinDecRan, GongLinNiu:ZClass, GongLinNiu:ZClass2, TWW, CarGabeSchafTikWhi2023}). We call such $C^*$-algebras \emph{classifiable}. An introduction to the classification program can be found in \cite{Strung:2021}. 

To complement the abstract classification theorem it is necessary to find a wide variety of ``naturally-occurring'' classifiable $C^*$-algebras. In fact the theory of $C^*$-algebras, starting with the seminal works of Murray and von Neumann, has a long and fruitful history of constructing explicit examples out of other mathematical objects such as dynamical systems or graphs. 

 In \cite{partI}, the author constructed $C^*$-algebras from topological as well as dynamical data, namely from a locally compact Hausdorff space $X$, a partial automorphism $\theta$ on $X$ (which is the same thing as a partial action of the integers on $X$ in the sense of \cite{Exel:2017}), and a vector bundle over the domain of $\theta$. This generalizes a construction from \cite{adamo2023} which uses global actions (that is, a homeomorphism $X \to X)$ instead of partial actions. Roughly speaking, we build a $C^*$-correspondence over $C_0(X)$ from the vector bundle and the partial automorphism. We then consider the associated Cuntz--Pimsner algebra as introduced by Pimsner \cite{Pimsner:1997} and further developed by Katsura \cite{Katsura:2004}.
 
 As for any class of examples of $C^*$-algebras, a natural question presents itself: 
 \begin{enumerate}
     \item When are the $C^*$-algebras in this class classifiable?
 \end{enumerate}
If the answer is positive, then we have the follow-up question:
 \begin{enumerate}[resume]
\item What is the range of the Elliott invariant for this class?
\end{enumerate}
It can be desirable to exhaust the range of the Elliott invariant by a particular class of examples. In this case, the classification theorem implies that every classifiable $C^*$-algebra can be obtained via such a model. This can provide further tools with which to study  these $C^*$-algebras, leading to new insights into their structure. For example, new models can uncover new Cartan subalgebras, or the existence of group actions and other symmetries. A celebrated result in this direction was achieved by Katsura in \cite{Katsura2007}, where it was shown that every Kirchberg algebra, meaning every simple separable nuclear purely infinite $C^*$-algebra, can be realised as a topological graph $C^*$-algebra. For similar constructions see for example \cite{Spielberg:2007, Wu:2025}.

While the class of $C^*$-algebras that is the focus of this paper---the $C^*$-algebras associated to partial automorphisms twisted by vector bundles---have non-trivial intersection with the class of Katsura's topological graph $C^*$-algebras, they are in many ways easier to work with.  As Cuntz--Pimsner algebras, they admit a canonical gauge action. In \cite{partI}, the author showed that there is a particularly satisfying description of the fixed point algebra: it can be described as the section algebra of a continuous field of either matrix or UHF algebras, depending on the partial action (see \cite[Theorem 3.28]{partI}). This allows for greater control over the conditional expectation onto the fixed point algebra, an indispensable technical tool in many proofs. For example, it was used in \cite{partI} to describe the closed two-sided ideals of a $C^*$-algebra associated to a partial automorphism twisted by a vector bundle. In Section \ref{sect:traces cp algebra} it also plays a crucial role in determining the tracial state space when the vector bundle has rank greater than one. This detailed understanding of the fixed point algebra will undoubtedly prove to be useful in many other situations as well. 

The construction also allows for many interesting stably finite $C^*$-algebras. From the point of view of the classification program, stably finite and purely infinite $C^*$-algebras have historically been treated separately, often using different techniques. The classification of Kirchberg algebras was already  completed in the early 1990's \cite{phillips:2000, kirchbergphillips:2000, kirchberg:2022}, well before the analogous results for stably finite $C^*$-algebras.  For this reason, the $C^*$-algebras presented here should be of great interest for the classification program, particularly in light of Questions (1) and (2) above. This paper is mainly concerned with answering Question (1) by computing bounds for the nuclear dimension. Sufficient conditions for finite nuclear dimension lead to the main result, Theorem \ref{thm:classifiable}:  A Cuntz--Pimsner algebra associated to a partial automorphism on $X$ twisted by a vector bundle, is classifiable whenever $X$ is an infinite compact second countable Hausdorff space with finite covering dimension and the partial automorphism is minimal. 

Why do we consider partial actions, instead of staying within the global framework of \cite{adamo2023}? A key motivation for this is provided by the results of Deeley, Putnam and Strung in \cite{deeleyputnamstrung:2018jiangsu, deeleyputnamstrung:2024}. They explored the range of the Elliott invariant of crossed products by minimal homeomorphisms and their orbit-breaking subalgebras. To break the orbit of a minimal homeomorphism means to choose a closed subset of the space which meets every orbit at most once. Then two points are equivalent under the new, orbit-breaking relation, whenever they can be connected via an orbit which does not meet the closed subset. The orbit-breaking subalgebra of the original crossed product is the groupoid $C^*$-algebra associated to the orbit-breaking equivalence relation. This is a special case of a partial crossed product by the integers, and so fits into the framework of the paper. As pointed out in \cite{deeleyputnamstrung:2018jiangsu, deeleyputnamstrung:2024}, all AF algebras as well as the Jiang--Su algebra can be realized as orbit-breaking subalgebras. This is a remarkable result because it allows for a dynamical presentation of both AF algebras and the Jiang--Su algebra (among other examples), despite the fact that none of these algebras can be obtained from crossed products by minimal homeomorphisms, due to $K$-theoretical obstructions. Twisting the partial automorphism associated to an orbit-breaking algebra by a vector bundle allows for the possibility of realizing even further invariants than those in \cite{deeleyputnamstrung:2024}.

 Partial actions and their crossed products are discussed in detail in \cite{Exel:2017}, and their classifiability under appropriate hypothesis was shown in \cite{Geffen2021}. Together with the results of \cite{deeleyputnamstrung:2018jiangsu, deeleyputnamstrung:2024} this suggests that when trying to construct as many classifiable $C^*$-algebras as possible, while staying in a dynamical framework, we ought to consider partial actions.

Crossed products and their orbit-breaking subalgebras from \cite{deeleyputnamstrung:2024} fall into the stably finite category. Purely infinite classifiable $C^*$-algebras are exactly the unital Kirchberg algebras which were constructed in \cite{Katsura2007} as topological graph $C^*$-algebras. In \cite{adamo2023} it is shown that in the case of a homeomorphism (rather that the partial automorphisms considered here), the resulting Cuntz--Pimsner algebras can either be stably finite or purely infinite depending on the rank of the vector bundle \cite[Corollary 5.6]{adamo2023}, thus placing both classes on an equal footing.  More precisely, they show that their $C^*$-algebras have tracial states if and only if the vector bundle is a line bundle, that is, all its fibers are one-dimensional. This implies that the algebras are stably finite in the line bundle case, and otherwise are purely infinite. 

One might expect a similar dichotomy when generalizing to the case of partial automorphisms twisted by vector bundles, as considered in \cite{partI} as well as this paper. Surprisingly, there are many examples associated to higher rank bundles for which the Cuntz--Pimsner algebras are stably finite. Indeed, if the vector bundle has rank $d\geq 1$ and the action is free, then tracial states are in bijection with so-called $d$-conformal measures on $X$, see Proposition \ref{prop:traces cp algebra}. A $1$-conformal measure is simply an invariant measure for the partial action. If the action is global and $d>1$, then one can easily see that there do not exist any $d$-conformal measures, thus recovering \cite[Proposition 4.4]{adamo2023}. However, there are many partial actions for which such measures do exist even for $d>1$. Notably, this occurs if the partial action arises from breaking the orbit of a global action, as described above. We are thus able to prove that such orbit-breaking subalgebras always possess traces, and therefore are always stably finite. This is quite remarkable since it was expected that, as for line bundles  \cite[Theorem 6.14]{adamo2023},  the orbit-breaking subalgebra would be a large subalgebra (in the sense of Phillips \cite{phillips:2014}) inside the original Cuntz--Pimsner algebra of the corresponding global action. However, under the assumption that the dimension of the unerlying space is finite, the ``global algebra'' is purely infinite while the orbit-breaking subalgebra is stably finite, which rules out the possibility that it is a large subalgebra.

Having established sufficient conditions for classification, future work aims to address Question (2), namely determining the range of the Elliott invariant. This will involve concrete examples of spaces, actions and vector bundles for which the Elliott invariant of the resulting algebras can be calculated explicitly.  

\subsection{Summary of the paper}
Section \ref{sect:preliminaries} contains the necessary preliminaries. We briefly introduce the theory of partial actions as well as the Cuntz--Pimsner construction. We recall some results from \cite{partI} concerning the structure of the fixed point algebra under the canonical gauge action, and say a few words about the Elliott classification program.

In Section \ref{sect: partial actions}  we prove some apparently new results about invariant measures and minimality of partial $\Z$-actions. We introduce conformal measures, and give necessary and sufficient conditions for their existence.

In Section \ref{sect:traces cp algebra} we show that if the action is free, then tracial states on our algebras are in bijection with conformal measures. 
 
Section  \ref{sect:classifiability of the Cuntz--Pimsner algebra} contains the main result of this paper, namely Theorem \ref{thm:classifiable}. This theorem shows classifiability of these Cuntz--Pimsner algebras under certain conditions, the most important of which are finite covering dimension of the base space, as well as minimality of the action. Since simplicity holds in this setting by \cite[Corollary 4.17]{partI}, and other properties such as nuclearity and the UCT follow from the general theory of Cuntz--Pimsner algebras, the bulk of the section is devoted to showing finite nuclear dimension. We use methods from \cite{Hirshberg2017} to show that in the special case that the partial automorphism only has orbits of a fixed size, the Cuntz--Pimsner algebra is homogeneous and has decomposition rank bounded by the dimension of the space. We then apply methods from \cite{Geffen2021} to show finite nuclear dimension. The main theorem follows. 

In the final section, we turn to orbit-breaking subalgebras. We first investigate the circumstances under which restricting partial automorphisms preserves minimality and the space of invariant measures. It turns out, see Proposition \ref{thm:blabla}, that on a compact space this is exactly the case when the set one takes out meets every orbit at most once. Since orbit-breaking subalgebras are special cases of our Cuntz--Pimsner algebras associated to partial automorphisms, several statements from \cite{phillips:2007} and \cite{adamo2023} can easily be generalized using the main results. 

\textit{Acknowledgements:} I would like to thank my advisor Karen Strung for suggesting this project, and for many useful and stimulating discussions. I would also like to thank Noa Bihlmaier and Rainer Nagel for helpful discussions regarding invariant measures of (partial) actions.

\section{Preliminaries}\label{sect:preliminaries}
We start this section by defining partial automorphisms. We then recall the theory of Hilbert modules and $C^*$-correspondences as well as the Cuntz--Pimsner construction. We introduce the $C^*$-correspondences that we will be working with, which are constructed from a partial automorphism as well as a vector bundle. Such correspondences were first considered in \cite{partI}, and in \cite{adamo2023} for global instead of partial actions. Finally we give a quick introduction into the classification program for $C^*$-algebras. 

\subsection{Partial actions}

Throughout this as well as the next sections $X$ will be a locally compact Hausdorff space unless explicitly stated otherwise. 
\begin{definition}
A \emph{partial automorphism on $X$} is a homeomorphism $\theta:U\to V$ between two open subsets $U$ and $V$ of $X$. For every $n\in\Z$, the maximal domain of the $n$-th power of $\theta$ will be denoted by $D_{-n}$, and its codomain by $D_n$. That is, $\theta^n$ maps $D_{-n}$ to $D_n$.
\end{definition}
 By definition, we have $D_{-1}=U$ and $D_1=V$. Furthermore the equalities 
\[D_n=\theta^{n-1}(D_{1-n}\cap D_1)\quad\text{and}\quad D_{-n}=\theta^{1-n}(D_{n-1}\cap D_{-1})\]
hold for all $n\geq 1$. They are special cases of the identity 
\[\theta^n(D_{-n}\cap D_{k})=D_n\cap D_{k+n}\]
which holds for all integers $n$ and $k$.
\begin{remark}
    Partial automorphisms fit into the wider framework of partial actions of groups, which in the context of $C^*$-algebras were introduced in \cite{Exel:1994, McClanahan:1995, Exel:1998}. In fact, a partial automorphism on $X$ is precisely a partial action on $X$ by the integers. There is a beautiful interplay between partial actions, inverse semigroups and Fell bundles, which in large parts was developed by Ruy Exel and collaborators (see \cite{Exel:2017}). 
\end{remark}

\begin{definition}
    Let $\theta$ be a partial automorphism on a locally compact Hausdorff space $X$. The \emph{orbit} of a point $x\in X$ is the set $\orb(x)\coloneqq\{\theta^n(x):n\in\Z\;\text{such that }x\in D_{-n}\}$. 

A subset $Y$ of $X$ is called \emph{$\theta$-invariant} if 
\begin{align*}
\theta(Y\cap U)\subset Y\quad\text{and}\quad\theta^{-1}(Y\cap V)\subset Y.
\end{align*} 
We call $\theta$ \emph{minimal} if the only closed (equivalently, open) $\theta$-invariant subsets are $\emptyset$ and $X$. We call $\theta$ \emph{free} if for all $x\in X$ the equality $\theta^n(x)=x$ implies $n=0$. 
\end{definition}

\subsection{$C^*$-correspondences and Cuntz--Pimsner algebras} We briefly review the necessary background on Hilbert modules, $C^*$-correspondences and Cuntz--Pimsner algebras. For a more detailed account see \cite{Lance:1995} and \cite{Katsura:2004}. 

Let $A$ be a $C^*$-algebra. Then a \emph{(right) Hilbert $A$-module} is a right $A$-module $\E$ equipped with an $A$-valued inner product $\langle\cdot,\cdot\rangle$ such that $\E$ is complete with respect to the norm $\norm{\xi}=\sqrt{\norm{\langle\xi,\xi,\rangle}_A}$. We write $\calL(\E)$ for the $C^*$-algebra of adjointable linear operators on $\E$. The $C^*$-algebra of compact operators $\K(\E)$ is the subalgebra of $\calL(\E)$ obtained as the closed linear span of the set of rank-one operators $\theta_{\xi,\eta}$ with $\xi,\eta\in\E$. Here $\theta_{\xi,\eta}$ is the adjointable operator defined by 
\[\theta_{\xi,\eta}(\zeta)=\xi\langle\eta,\zeta\rangle\qquad\text{for }\xi,\eta,\zeta\in\E.\]
To avoid confusion with a partial automorphism $\theta$ we will sometimes write $\xi\eta^*$ instead of $\theta_{\xi,\eta}$.

A \emph{$B$-$A$ $C^*$-correspondence}, where $A$ and $B$ are $C^*$-algebras, is a right Hilbert $A$-module together with a $\ast$-homomorphism $\varphi_\E$ from $B$ to $\calL(\E)$. We call $\varphi_\E$ the \emph{left action} of the $C^*$-correspondence. We often omit $\varphi_\E$ and write $b\xi$ instead of $\varphi_\E(b)\xi$, for $b\in B$ and $\xi\in\E$. If $A$ and $B$ coincide then we say that $\E$ is a $C^*$-correspondence \emph{over} $A$.

Given a $C^*$-correspondence $\E$ over $A$, we can form a $C^*$-algebra $\CP(\E)$ called the \emph{Cuntz--Pimsner algebra of $\E$}. This construction was pioneered in \cite{Pimsner:1997} and generalized in \cite{Katsura:2004}. It is defined via a universal property for so-called covariant representations of $\E$, and can explicitly be constructed using the Fock space associated to $\E$. The details can be found in \cite{Katsura:2004}. The Cuntz--Pimsner algebra comes with the universal covariant representation $(\pi_u,t_u)$, where $\pi_u$ is an injective $\ast$-homomorphism from $A$ to $\CP(\E)$ and $t_u$ is an injective linear map from $\E$ to $\CP(\E)$. They fulfill the relations
\begin{align*}
\pi_u(a)t_u(\xi)\pi_u(b)=t_u(a\xi b)\quad \text{and}\quad t_u(\xi)^*t_u(\eta)=\pi_a(\langle\xi,\eta\rangle)
\end{align*}
for all $a,b\in A$ and $\xi,\eta\in\E$. We obtain a $\ast$-homomorphism $\psi_t$ from $\K(\E)$ into the Cuntz-Pimsner algebra, given on rank-one operators by sending $\theta_{\xi,\eta}$ to $t_u(\xi)t_u(\eta)^*$, for $\xi,\eta\in\E$. The covariance condition $\psi_t(\varphi_\E(a))=\pi_u(a)$ holds for all $a\in J_\E$, where 
\[J_\E\coloneqq \varphi_\E^{-1}(\K(\E))\cap (\ker\varphi_\E)^\perp.\]

  We will often omit the universal covariant representation, and regard $A$ and $\E$ as subsets of the Cuntz--Pimsner algebra. For example, we might write $\xi\eta^*$ instead of $t_u(\xi)t_u(\eta)^*$.

 There is a strongly continuous action of the circle on $\CP(\E)$ called the \emph{gauge action}, usually denoted by $\gamma$. It induces a topological $\Z$-grading in the sense of \cite[Definition 19.2]{Exel:2017} on $\CP(\E)$ with the $n$-th spectral subspace 
 \[\CP(\E)^n=\{a\in\CP(\E)|\forall z\in\mathbb{T}: \gamma_z(a)=z^na\},\] see \cite{Katsura:2004}. Of particular importance is $\CP(\E)^0$, the zeroth spectral subspace. It consists of all such elements of $\CP(\E)$ that are fixed under the gauge action, meaning $\gamma_z(a)=a$ for all $z\in\mathbb{T}$. For this reason we call $\CP(\E)^0$ the \emph{fixed point algebra}. Since the grading is topological, there exists a faithful conditional expectation from $\CP(\E)$ onto $\CP(\E)^0$, which we denote by $\Phi$. 

Let $\E^{\otimes n}$ denote the $n$-fold balanced tensor product of $\E$ with itself. For any $n\in\N$ there exists a subalgebra $B_n$ of $\CP(\E)^0$ isomorphic to $\K(\E^{\otimes n})$, obtained by setting
\[B_n\coloneqq \psi_{t^{\otimes n}}(\K(\E^{\otimes n})).\]
 For $n=0$ one defines $B_0$ to be equal to $\pi_u(A)$. Furthermore, we define
 \[B_{[0,n]}\coloneqq B_0+...+B_n.\]
Then $\{B_{[0,n]}\}_{n\geq 0}$ is an increasing sequence of subalgebras exhausting $\CP(\E)^0$, meaning that the closure of $\bigcup_{n\geq 0}B_{[0,n]}$ is equal to $\CP(\E)^0$ (see \cite[Proposition 5.7]{Katsura:2004}). In particular, we have
\[\CP(\E)^0=\lim\limits_{\to}B_{[0,n]}.\]

\subsection{Twisted partial automorphisms}\label{sect:twisted partial autos}
Let $X$ be a locally compact second countable Hausdorff space. We now define the $C^*$-correspondences we will be working with. The input data of the construction, which was first carried out in \cite{partI} and is inspired by a similar construction in \cite{adamo2023}, are a partial automorphism on $X$, as well as a vector bundle over the domain of the partial automorphism. 

 Let $\theta:U\to V$ be a partial automorphism on $X$ as in Section \ref{sect: partial actions}. Let $\V$ be a vector bundle over $U$. That is, $\V$ is a triple $[E,p,U]$ where $E$ is a topological space called the \emph{total space}, and $U$ is called the \emph{base space}. The \emph{bundle projection} $p$ is a continuous map from $E$ to $U$. For every $x\in U$, the fiber $p^{-1}(x)$ is a complex vector space. We always assume our vector bundles to be locally trivial, meaning that for every $x\in U$ there exists a neighborhood $W\subset U$ of $x$ such that $p^{-1}(W)$ is isomorphic, as vector bundles over $W$, to the trivial bundle $W\times\C^d$ for some $d\in\N$. The number $d$ is called the \emph{rank} of the bundle over $x$. It can depend on $x$, but is constant on connected components of $U$. For an introduction to the theory of vector bundles see \cite{Husemoller:1993}.

Let $\beta$ be a metric on $\V$, which exists and is unique up to isometry because $X$ is paracompact (see \cite[Chapter 3]{Husemoller:1993} and \cite[Theorem 2.5]{partI}). We can then consider $\Gamma_0(\V)$, the space of continuous sections of $\V$ vanishing at infinity with respect to $\beta$. There is a natural right Hilbert $C_0(X)$-module structure on $\Gamma_0(\V)$, with right multiplication and inner product given by 
\[\xi\cdot f(x)\coloneqq \xi(x)f(x),\qquad \langle\xi,\eta\rangle(x)=\beta(\xi(x),\eta(x)),\]
for $f\in C_0(X)$ and $\xi,\eta\in\Gamma_0(\V)$. 
Furthermore, we can define a $\ast$-homomorphism $\varphi$ from $C_0(X)$ into $\calL(\Gamma_0(\V))$ by setting
\[\varphi(f)\xi\coloneqq \xi(f\circ\theta)\]
for $f\in C_0(X)$ and $\xi\in\Gamma_0(\V)$. That is, left multiplication with $f\in C_0(X)$ is given by right multiplication with $f\circ\theta$. Note that the expression $f\circ\theta$ is not well defined on its own, unless $f$ vanishes outside of $V$. However, $\xi(f\circ\theta)$ is well defined since $\xi$ vanishes outside of $U$. 

We thus obtain a $C^*$-correspondence over $C_0(X)$, which we denote by $\Gamma_0(\V,\theta)$.


\subsection{Fibered structure of the fixed point algebra}\label{sect:fibered structure}

Let $X$ be a second countable locally compact Hausdorff space, $U$ and $V$ open subsets of $X$, and $\theta:U\to V$ a partial automorphism. Let $\V$ be a rank $d$ vector bundle over $U$. We write $\E\coloneqq\Gamma_0(\V,\theta)$.

Recall that the fixed point algebra $\CP(\E)^0$ contains all such elements of $\CP(\E)$ that are fixed under the canonical gauge action. It was shown in \cite[Proposition 3.22]{partI} that $\pi_u(C_0(X))$ is contained in the center of $\CP(\E)^0$, and that this turns $\CP(\E)^0$ into a \emph{$C_0(X)$-algebra}. A $C_0(X)$-algebra is a $C^*$-algebra $A$ together with a $\ast$-homomorphism $\phi$ from $C_0(X)$ into $\mathcal{Z}(M(A))$, the center of the multiplier algebra of $A$, such that $\phi(C_0(X))A$ is dense in $A$. For $x\in X$ the quotient of $A$ by the ideal $C_0(X\backslash\{x\})A$ is called the \emph{fiber} over $x$. Any $C_0(X)$-algebra is the section algebra of an upper-semicontinuous $C^*$-bundle over $X$, see \cite{Nilsen:1996}. 

We now describe the $C_0(X)$-structure. The $C^*$-algebra $\K(\E)$ of compact operators is a $C_0(X)$-algebra, where the map from $C_0(X)$ into the center of $M(\K(\E))=\calL(\E)$ is given by sending $f\in C_0(X)$ to $\varphi_\E(f)$. That $\varphi_\E$ indeed maps into the center of $\calL(\E)$ can be verified on rank-one operators. The fiber of $\K(\E)$ at a point $x\in V$ is given by $M_d$, and we have
\[\xi\eta^*(x)=\xi(\theta^{-1}(x))\eta(\theta^{-1}(x))^*\]
for $\xi,\eta\in\E$. If $x$ does not lie in $V$, then the fiber of $\K(\E)$ over $x$ is the trivial $C^*$-algebra $\{0\}$. Even though it seems more general, we obtain a $C_0(X)$-structure on $\K(\E^{\otimes n})$ for any $n\in\N$ as a special case. 

The $C_0(X)$-structure on $\K(\E^{\otimes n})$ yields a $C_0(X)$-structure on $B_{[0,n]}$ for all $n\in\N$, as well as on $\CP(\E)^0$. The details can be found in \cite[Section 3.5]{partI}. The fiber of $B_{[0,n]}$ over a point $x\in D_n$ is given by $M_{d^n}$. If $x\in X$ lies in $D_j$ but not in $D_{j+1}$, then the fiber over $x$ is $M_{d^j}$, for all $j=0,...,n-1$. The fibers of the fixed point algebra now follow from the fact that it is the inductive limit of the sequence $\{B_{[0,n]}\}_{n\geq 0}$: If a point $x$ lies in $D_j$ but not in $D_{j+1}$, then the fiber over $x$ is $M_{d^j}$, for all $j\geq 0$. If $x$ lies in the intersection $\bigcap_{j\geq 0}D_j$, then the fiber over $x$ is the UHF algebra $M_{d^\infty}$.

\subsection{The classification program for $C^*$-algebras}\label{sect:classification program}

Decades of intensive research by many mathematicians have led to a classifiation theorem for $C^*$-algebras that are unital, separable, simple, infinite dimensional, have finite nuclear dimension and satisfy the universal coefficient theorem (UCT). Such $C^*$-algebras are called \emph{classifiable}. For an overview see \cite{Strung:2021} and \cite{rordam:2002}. The theorem states that classifiable $C^*$-algebras are, up to isomorphism, classified by an invariant known as the Elliott invariant, which consists of $K$-theoretic as well as tracial data. 

The UCT is a condition relating $K$-theory and Kasparov's $KK$-theory. It is one of the major open questions in the field whether every separable nuclear $C^*$-algebra satisfies the UCT. In particular, the UCT is almost always fullfilled in examples, and checking it will not be an issue for us. 

This is very different for the other distinct property that classifiable $C^*$-algebras have, namely finite nuclear dimension. For the convenience of the reader, we recall the definition. 

\begin{definition}[{\cite{winterzacharias:2010}}]
A $C^*$-algebra $A$ has \emph{nuclear dimension} at most $n$, namely $\dim_{\mathrm{nuc}}(A)\leq n$, if there exists a net $(F_\lambda,\psi_\lambda,\varphi_\lambda)_{\lambda\in\Lambda}$ consisting of: $F_\lambda$ finite dimensional $C^*$-algebras, and $\psi_\lambda:A\to F_\lambda$, $\varphi_\lambda:F_\lambda\to A$ completely positive maps, satisfying: 
\begin{enumerate}
\item $\varphi_\lambda\circ\psi_\lambda(a)\to a$ uniformly on finite subsets of $A$;
\item $\psi_\lambda$ are contractive;
\item $F_\lambda$ decomposes as $F_\lambda=F_\lambda^{(0)}\oplus...\oplus F_\lambda^{(n)}$ such that the restriction of $\varphi_\lambda$ to each $F_\lambda^{(i)}$, $i\in\{0,...,n\}$, is a contractive order zero map.
\end{enumerate}
If one can moreover arrange that the maps $\varphi_\lambda$ are contractive, then we say that the \emph{decomposition rank} of $A$ is at most $n$, and write $\mathrm{dr}(A)\leq n$.
\end{definition}

We have chosen to include finite nuclear dimension rather than $\mathcal{Z}$-stability, that is, tensorial absorption of the Jiang-Su algebra $\mathcal{Z}$, because we will show classifiability of our $C^*$-algebras by finding an upper bound on the nuclear dimension.

A unital $C^*$-algebra $A$ is called \emph{finite} if $1_A$ is a finite projection, and \emph{stably finite} if every matrix amplification $M_n(A)$ is finite. A simple $C^*$-algebra $A$ is called \emph{purely infinite} if every hereditary subalgebra of $A$ contains an infinite projection. There is a dichotomy: A classifiable $C^*$-algebra is either purely infinite or stably finite. The latter is the case if and only if it has tracial states. It should be mentioned that the classification of purely infinite classifiable $C^*$-algebras, also called Kirchberg algebras, has already been completed in the early 1990's \cite{phillips:2000, kirchbergphillips:2000, kirchberg:2022}. 

\section{Partial automorphisms}\label{sect: partial actions}
This section contains basic results about minimality and invariant measures of partial automorphisms. It also treats what we call conformal measures, which will play in important role in Section \ref{sect:tracial state space}.
\subsection{Decomposing the space into orbit types}\label{sect:orbit types}

Let $\theta:U\to V$ be a partial automorphism on a locally compact Hausdorff space $X$.
For any $x\in X$ exactly one of the following holds. 
\begin{enumerate}
\item Both $\theta$ and $\theta^{-1}$ can only be applied finitely many times to $x$. That is, we have $x\notin D_{-M}$  and $x\notin D_N$ for some $M,N\in\N$.
\item We can apply $\theta$ infinitely often to $x$, but $\theta^{-1}$ only finitely many times. That is, we have $x\in D_{-n}$ for all $n\in\N$ and $x\notin D_N$ for some $N\in\N$. 
\item We can apply $\theta^{-1}$ infinitely often to $x$, but $\theta$ only finitely many times. That is, we have $x\in D_{n}$ for all $n\in\N$ and $x\notin D_{-N}$ for some $N\in\N$.
\item We can apply both $\theta$ and $\theta^{-1}$ infinitely often to $x$. That is, we have $x\in D_n$ for all $n\in\Z$.
\end{enumerate}
 Denote the set of all points fulfilling (1), (2), (3) or (4) by $D_{\fin}$, $D_+$, $D_-$ and $D_{\gl}$, respectively. Note that $D_{\fin}$ includes the set $X\backslash(U\cup V)$ of all points on which $\theta$ does not act at all. Also note that $D_{\gl}$ contains points with periodic orbits, meaning that points in $D_{\gl}$ do not necessarily have an infinite orbit. 
Explicitly, we have
\begin{align*}
D_{\fin}&=\bigcup_{n\geq0}D_n\backslash D_{n+1}\cap\bigcup_{n\geq0}D_{-n}\backslash D_{-n-1},\\
D_+&=\left(\bigcap_{n>0}D_{-n}\right)\backslash\left(\bigcap_{n\in\Z}D_n\right) ,\\
D_-&=\left(\bigcap_{n>0}D_n\right)\backslash \left(\bigcap_{n\in\Z}D_n\right),\\
D_{\gl}&=\bigcap_{n\in\Z}D_n.
\end{align*} 
This implies that $D_{\fin}$, $D_+$, $D_-$ and $D_{\gl}$ are all measurable. They are mutually disjoint and their union is equal to $X$.

\subsection{Invariant measures and minimality}\label{sec:invariant measures}

Let $X$ be a locally compact Hausdorff space, $U$ and $V$ open subsets of $X$, and $\theta:U\to V$ a partial automorphism on $X$. We do not require $X$ to be infinite. We write $M^1(X)$ for the compact convex set of regular Borel probability measures on $X$. A measure $\mu\in M^1(X)$ is called \emph{$\theta$-invariant} if $\mu(\theta^{-1}(Y))=\mu(Y)$ holds for all measurable sets $Y\subset V$. The set of all such measures is denoted by $M^1_\theta(X)$. 

\begin{proposition}\label{prop: no measures}
If $\mu\in M_\theta^1(X)$ then $\mu(D_+)=\mu(D_-)=0$.
\end{proposition}
\begin{proof}
We will prove $\mu(D_+)=0$, that $\mu(D_-)=0$ can be shown analogously. Intuitively the proof works in the following way: Each orbit in $D_+$ has a starting point, which is the unique point in the orbit not lying in $D_1$. We collect all starting points in the set $Y=D_+\backslash D_1$. Then $Y$ is measurable because $D_+$ and $D_1$ are. The sets $\theta^n(Y)$ for $n\geq0$ are all mutually disjoint, but must have the same measure by the $\theta$-invariance of $\mu$, and their union is equal to $D_+$. Since $\mu$ is a probability measure this concludes the proof. 

Now the technical part: A short calculation shows that $\theta^n(D_+)$ equals $D_+\cap D_n$, and thus 
\begin{align*}
   \theta^n(Y)&=\theta^n(D_+\backslash(D_1\cap D_{-n}))=\theta^n(D_+)\backslash\theta^n(D_1\cap D_{-n})=(D_+\cap D_n)\backslash D_{n+1}\\&=D_+\cap(D_n\backslash D_{n+1}).
\end{align*}
 This implies $\theta^n(Y)\cap\theta^m(Y)=\emptyset$ for all $n\neq m$, $n,m\geq0$. Furthermore we obtain
\begin{align*}
\bigcup_{n\geq0}\theta^n(Y)&=D_+\cap\bigcup_{n\geq0}D_n\backslash D_{n+1}=D_+.
\end{align*}
Hence
\begin{align*}
\mu(D_+)=\mu(\bigcup_{n\geq0}\theta^n(Y))=\sum_{n=0}^\infty\mu(\theta^n(Y))= \sum_{n=0}^\infty\mu(Y),
\end{align*}
which implies $\mu(Y)=0$, and therefore $\mu(D_+)=0$.
\end{proof}
\begin{corollary}\label{cor: no nonempty invariant subset}
 The set $D_+\cup D_-$ cannot contain a nonempty compact $\theta$-invariant subset.
\end{corollary}
\begin{proof}
    If $K$ is a compact $\theta$-invariant set contained in $D_-\cup D_+$ then by \cite[Theorem 4.21]{partI} there exists a probability measure $\mu$ such that $\supp\mu\subset K$. By Proposition \ref{prop: no measures} we must have $\mu(K)=0$, which is a contradiction. 
\end{proof}
\begin{remark}
    One can also prove Corollary \ref{cor: no nonempty invariant subset} without using invariant measures. For this one first proves that for a minimal partial automorphism on a compact space, the set $D_{\gl}$ has to be dense in $X$, so in particular it is nonempty. One then shows that if $X$ is compact, then there exists a closed invariant subset such that the action restricted to this subset is minimal (which is well known in the global case, see for example \cite[Theorem 3.5]{efhn:2016}). The minimal subsystem must have a global orbit. This orbit is global for the original action as well, which concludes the proof. 
\end{remark}

Proposition \ref{prop: no measures} could be interpreted as saying that from a measure-theoretical perspective,  $D_-$ and $D_+$ can be ignored. By choosing an invariant measure we can pass from a partial automorphism, which is a topological dynamical system, to a measure theoretical dynamical system. Then the resulting measure preserving transformation will be defined almost everywhere, at least as long as $D_{\fin}$ is empty. It seems surprising that $D_-$ and $D_+$, which can make up a considerable part or (if $X$ is not compact) possibly all of the space, should not play any role when it comes to measure theoretic aspects.  Moreover, this phenomenon is not restricted to measure theory, as the following proposition shows. This proposition is an analogue of a well-known theorem concerning minimality of $\Z$-actions (see for example \cite[Lemma 11.1.2]{Giordano2018}), and tells us that on compact spaces, $D_-$ and $D_+$ can be ignored when considering minimality. 
\begin{theorem}\label{thm:minimality partial case}
    If $X$ is a locally compact Hausdorff space and $\theta:U\to V$ is a partial automorphism, then the following are equivalent:
    \begin{enumerate}
        \item The partial automorphism $\theta$ is minimal.
        \item For every point $x$ in $X$ the orbit $\orb(x)$ is dense in $X$.
        \end{enumerate}
    If $X$ is compact and infinite and $D_{\fin}$ is empty then (1) as well as (2) are equivalent to each of the following statements: 
    \begin{enumerate}
            \item[(3)] For every point $x$ in $D_{\gl}$ the orbit $\orb(x)$ is dense in $X$.
         \item[(4)]  For every point $x\in D_{\gl}$ the forward orbit $\orb_+(x)$ is dense in $X$.
         \item[(5)]  For every point $x\in D_{\gl}$ the backward orbit $\orb_-(x)$ is dense in $X$.
         \item[(6)] If $Y$ is a closed $\theta$-invariant subset such that $Y\cap D_{\gl}\neq\emptyset$, then $Y\in\{\emptyset,X\}$. 
       \item[(7)]  If $Y$ is a closed set such that $Y\cap D_{\gl}\neq\emptyset$ and $\theta(Y\cap U)\subset Y$, then $Y\in\{\emptyset,X\}$.
        \item[(8)]   If $Y$ is a closed set such that $Y\cap D_{\gl}\neq\emptyset$ and $\theta^{-1}(Y\cap V)\subset Y$, then $Y\in\{\emptyset,X\}$. 
    \end{enumerate}
\end{theorem}
The proof of Theorem \ref{thm:minimality partial case} is very similar to the global case. The only additional ingredient is that by Corollary \ref{cor: no nonempty invariant subset}, if $Y$ is a closed $\theta$-invariant subset of $X$ then $Y\cap D_{\gl}\neq\emptyset$.
\begin{example}
Let $X=\mathbb{T}$ and let $\alpha:\mathbb{T}\to\mathbb{T}$ be the rotation on the torus by an irrational angle. Choose a point $x\in\mathbb{T}$. We define $\theta:\mathbb{T}\backslash\{x\}\to\mathbb{T}\backslash\{\alpha(x)\},y\mapsto\alpha(y)$. Then $\theta$ is minimal since every orbit is dense. We have
\begin{align*}
D_+=\orb_+(\alpha(x)),\quad D_-=\orb_-(x),\quad\text{and}\quad  D_{\gl}=X\backslash\orb(x).
\end{align*}
The only $\theta$-invariant measure is the Lebesgue measure on the torus.  
\end{example}

\subsection{Conformal measures}

\begin{definition}
Let $\theta:U\to V$ be a partial automorphism on a locally compact Hausdorff space $X$. A probability measure $\mu\in M^1(X)$ is called \emph{$d$-conformal with respect to $\theta$} for some $d\in\N$ if 
\[\mu(\theta(Y))=d\mu(Y)\]
for all measurable $Y\subset U$. This is a special case of a more general definition of conformal measures, see \cite[Definition 2.1]{DenkerUrbanski:1991}.

We write $M^1_\theta(X,d)$ for the compact convex set of all probability measures on $X$ that are $d$-conformal with respect to $\theta$.

Conformal measures for $d=1$ are the same as $\theta$-invariant measures in the sense of Section \ref{sec:invariant measures}.
\end{definition}

\begin{proposition}\label{prop:conformal measures}
Let $X$ be a locally compact Hausdorff space, and take $d>1$.
For $\mu\in M_\theta^1(X,d)$ we have 
\[\mu(D_{\gl})=\mu(D_+)=0. \]
 If $X$ is compact and either $D_-$ or $D_{\fin}$ is nonempty, then $M_\theta^1(X,d)$ is nonempty as well. 
\end{proposition}
\begin{proof}
 Let $\mu$ be $d$-conformal with respect to $\theta$.  We have $D_{\gl}=\theta(D_{\gl})$, and hence
\[\mu(D_{\gl})=\mu(\theta(D_{\gl}))=d\mu(D_{\gl})\]
which implies $\mu(D_{\gl})=0$ because we assume $d>1$. 

To show $\mu(D_+)=0$ is similar to the proof of Proposition \ref{prop: no measures}. We can write $D_+=\bigsqcup_{i=0}^\infty\theta^{i}(Y)$ for some measurable set $Y$ contained in $D_+$. We then have 
\[\mu(D_+)=\sum_{i=0}^\infty d^i\mu(Y),\]
the right hand side of which can only converge if $\mu(Y)$, and hence $\mu(D_+)$, vanishes. 

Now suppose that $X$ is compact. We show that $M_\theta^1(X,d)$ is nonempty if $D_-$ is nonempty. This works similarly to how one would prove existence of an invariant measure. Take $x\in D_-$, and assume without loss of generality that $x$ does not lie in $U$. 

For every $n\in\N$ define $\mu_n\in M^1(X)$ by
\[\mu_n\coloneqq\left(\sum_{i=0}^nd^{-i}\right)^{-1}\left(\sum_{i=0}^nd^{-i}\delta_{\theta^{-i}(x)}\right)\]
where $\delta_{\theta^{-i}(x)}$ is the Dirac measure centered at $\theta^{-i}(x)$. We have $\norm{\mu_n}=1$ for all $n\in\N$, and thus the sequence $\{\mu_n\}_{n\in\N}$ has a subsequence converging to a measure $\mu\in M^1(X)$ in the weak*-topology. 
For any measurable subset $Y$ of $U$ we have
\[ |\mu_n(\theta(Y))-d\mu_n(Y)|\leq \mu(\theta^n(x))=d^{-n}\left(\sum_{i=0}^nd^{-i}\right)^{-1}\]
and the right hand side converges to zero as $n$ goes to infinity. Thus we obtain $\mu(\theta(Y))=d\mu(Y)$, which shows that $\mu$ lies in $M^1_{\theta}(X,d)$. The case that $D_f\neq\emptyset$ is similar, hence omitted. 
\end{proof}

\section{The tracial state space}\label{sect:tracial state space}

The main purpose of this section is to investigate the tracial state space of the Cuntz--Pimsner algebras associated to our correspondences. This was initiated in \cite{partI} by studying the case of line bundles. We now generalize the result obtained there to higher rank bundles.

\subsection{Traces on the fixed point algebra}

Let $X$ be a second countable locally compact Hausdorff space, $U$ and $V$ open subsets of $X$, and $\theta:U\to V$ a partial automorphism on $X$. Let $\V$ be a rank $d$ vector bundle on $X$. As in Section \ref{sect:twisted partial autos} we write $\E$ for the $C^*$-correspondence $\Gamma_0(\V,\theta)$ over $C_0(X)$.  

In the following, whenever we regard $M_{d^m}$ as a subalgebra of $M_{d^n}$ for natural numbers $d$, $m$ and $n$ with $m<n$, then the embedding is given by sending $a\in M_{d^m}$ to $a\otimes 1_{d^{n-m}}$.  

\begin{remark}\label{rem:traces for cont fields of matrices}
Take $n\in\N$, and let $A$ be a separable $m$-homogeneous $C^*$-algebra with primitive ideal space $X$. Tracial states on $A$ are in bijection to regular Borel probability measures on $X$. Recall from \cite{Fell:1961} that $A$ is a locally trivial $C_0(X)$-algebra with fibers $M_m$. There is a unique tracial state $\tr_m$ on $M_m$, given by the usual normalised trace. Then every trace $\tau$ on $A$ is of the form
\[\tau(a)=\tau_\mu(a):=\int_X\tr_m(a(x))d\mu,\qquad a\in A,\]
for some probability measure $\mu$ on $X$. If $\tau$ is a non-normalised trace, then $\mu$ is not a probability measure (but still finite). 
\end{remark}

\begin{lemma}\label{lem:traces interval algebras}
For every $n\geq 0$ there is an affine homeomorphism between $M^1(X)$ and $T(B_{[0,n]})$, given by sending a probability measure $\mu$ to 
\[ \tau_\mu(a)=\int_X\tr_{d^n}(a(x))d\mu(x),   \]
for $a\in B_{[0,n]}$. 
\end{lemma}
\begin{proof}
By definition we have $B_{[0,n]}=B_0+...+B_n$ with $B_0\cong C_0(X)$ and $B_i\cong \K(\E^{\otimes i})$ for $i=1,...,n$. Recall from Section \ref{sect:fibered structure} that $\K(\E^{\otimes i})$ is a homogeneous $C^*$-algebra with fibers $M_{d^n}$ (where $d$ is the rank of the vector bundle) and primitive ideal space homeomorphic to $D_{i}$. Here $D_i$ is the $i$-th domain of $\theta$. 

Take $\tau\in T(B_{[0,n]})$. According to Remark \ref{rem:traces for cont fields of matrices} there exist finite measures $\mu_i$ on $D_i$ for $i=0,...,n$ such that $\tau$ restricted to $B_i$ is equal to $\tau_{\mu_i}$.  Recall that for $i<j$, the domain $D_{j}$ is contained in $D_{i}$. The restriction $\mu_i|_{D_j}$ of $\mu_i$ to $D_j$ agrees with $\mu_j$. Thus we can define a measure $\mu$ on $X$ by setting $\mu|_{D_i}\coloneqq \mu_i$. Any $a\in B_{[0,n]}$ can be written as a sum of $a_i\in B_i$ for $i=0,...,n$. The calculation 
\begin{align*}
\tau(a)=\sum_{i}\tau(a_i)=\sum_{i}\int_{D_i}\tr_{d^n}(a_i(x))d\mu_i=\sum_{i}\int_X\tr_{d^n}(a_i(x))d\mu=\int_X\tr_{d^n}(a(x))d\mu
\end{align*}
yields the statement of the lemma. 
\end{proof}
Recall that the fiber $\CP(\E)^0_x$ for $x\in X$ is isomorphic either to the trivial $C^*$-algebra, to a matrix algebra or to the UHF algebra $M_{d^\infty}$ of type $d^\infty$. We write $\tr_{d^\infty}$ for the unique normalised trace on $M_{d^\infty}$. 
\begin{proposition}\label{prop:traces fp algebra}
There is an affine homeomorphism 
\[\rho:M^1(X)\to T(\CP(\E)^0),\]
 given by sending a probability measure $\mu$ to 
\[ \tau_\mu(a)=\int_X\tr_{d^\infty}(a(x))d\mu(x),   \]
for $a\in \CP(\E)^0$.
\end{proposition}
\begin{proof}
Recall that $\CP(\E)^0$ is the inductive limit of the $B_{[0,n]}$ with the obvious inclusions as connecting maps, which are $C_0(X)$-linear and unital. Let $\tau$ be a trace on $\CP(\E)^0$. Restricting $\tau$ to $B_{[0,n]}$ and applying Lemma \ref{lem:traces interval algebras} yields a sequence of probability measures $\{\mu_n\}_{n\geq 0}$ such that $\tau|_{B_{[0,n]}}=\tau_{\mu_n}$. Any $f\in C_0(X)$ lies in $B_{[0,n]}$ for all $n\geq 0$. We have $\tau_{\mu_m}(f)=\tau_{\mu_n}(f)$ and thus $\mu_m=\mu_n$  for all $m,n\geq 0$. Write $\mu$ for $\mu_n$. 

Take $a\in\CP(\E)^0$ and $\varepsilon>0$. There exist $n\geq 0$ and $a_n\in B_{[0,n]}$ such that $\norm{a-a_n}<\varepsilon/2$. Then 
\begin{align*}
&|\tau(a)-\int_X\tr_{d^\infty}(a(x))d\mu(x)|\leq |\tau(a)-\tau_{\mu}(a_n)|+|\tau_{\mu}(a_n)-\int_X\tr_{d^\infty}(a(x))d\mu(x)|\\
&= |\tau(a)-\tau(a_n)|+\int_X|\tr_{d^n}(a_n(x))-\tr_{d^\infty}(a(x))|d\mu(x)\leq \norm{a-a_n}+\norm{a_n-a}<\varepsilon.
\end{align*}
 Since $\varepsilon$ was arbitrary, this shows the claim. 
\end{proof}

\subsection{Traces on the Cuntz--Pimsner algebra induced by conformal measures}\label{sect:traces cp algebra}
We now prove the main result of this section, namely that traces on the Cuntz--Pimsner algebra associated to $\Gamma_0(\V,\theta)$ are in bijective correspondence with certain probability measures on $X$. We first need to introduce such measures, and investigate conditions for their existence.  

The setting is the same as in the previous section. In particular our $C^*$-correspondence $\E$ is given by $\Gamma_0(\V,\theta)$, for a partial automorphism $\theta$ on the second countable locally compact Hausdorff space $X$, and a vector bundle $\V$ of constant rank $d\in\N$.

\begin{lemma}\label{lem:restricting traces}
Assume that $\theta$ is free. Let $\Phi:\CP(\E)\to \CP(\E)^0$ be the canonical conditional expectation. Then the map $\pi:T(\CP(\E))\to T(\CP(\E)^0)$ given by restricting traces is an affine homeomorphism onto its image. Its inverse sends a trace $\tau$ on $\CP(\E)^0$ to $\tau\circ\Phi$. 
\end{lemma}
\begin{proof}
Using \cite[Lemma 4.13]{partI} one can show that $\tau\circ\Phi=\tau$ for all $\tau\in T(\CP(\E))$ in the same way as in the proof of \cite[Proposition 4.23]{partI}. Thus any trace on $\CP(\E)$ is completely determined by its restriction to $\CP(\E)^0$. The claim follows.
\end{proof}

Recall that $\pi$ is the map from $T(\CP(\E))$ to $T(\CP(\E)^0)$ given by Lemma \ref{lem:restricting traces}, and $\rho$ is the affine homeomorphism from $T(\CP(\E)^0)$ to $M^1(X)$ given by Proposition \ref{prop:traces fp algebra}. 
\begin{proposition}\label{prop:traces cp algebra}
 We have $\im\rho\circ\pi=M_\theta^1(X,d)$, in particular there is an affine homeomorphism between $T(\CP(\E))$ and $M_\theta^1(X,d)$ if $\theta$ is free. The tracial state on $\CP(\E)$ associated to a $d$-conformal probability measure $\mu$ is given by 
\[\tau_\mu(a)=\int_X\tr_{d^\infty}(\Phi(a)(x))d\mu(x)\]
for all $a\in\CP(\E)$. 
\end{proposition}
\begin{proof}
We first show that the image of $\rho\circ\pi$ contains $M_\theta^1(X,d)$. Recall that $\CP(\E)^k$ denotes the $k$-th spectral subspace of the Cuntz--Pimsner algebra.

 Let $\mu$ be a probability measure that is $d$-conformal with respect to $\theta$. We have to show that $\tau\coloneqq\tau_\mu\circ\Phi$ is a trace on $\CP(\E)$. One can verify that it is a state. Thus it remains to show that $\tau(ab)=\tau(ba)$ for all $a,b\in\CP(\E)$. 

As a first step, we show that for all $a_k,b_k\in\CP(\E)^k$, we have $\tau_\mu(b_k^*a_k)=\tau_\mu(a_kb_k^*)$. 
To this end, take $\xi,\eta\in\E^k$. By definition of the $C_0(X)$-structure on $\CP(\E)^0$ we have 
\[\xi\eta^*(x)=\xi(\theta^{-k}(x))\eta(\theta^{-k}(x))^*\]
for all $x\in D_k$, and $\xi\eta^*(x)=0$ for $x\notin D_k$. 

If $V$ is any finite dimensional complex vector space with normalised trace $\tr_V$ on $\End(V)$ and $v,w\in V$, then $w^*v=\tr_V(w^*v)=\dim(V)\tr_V(vw^*)$. Thus we have
\begin{align*}\tau_\mu(\xi\eta^*)&=\int_X \tr_{d^k}(\xi\eta^*(x))d\mu(x)=\int_X\tr_{d^k}(\xi(\theta^{-k}(x))\eta(\theta^{-k}(x))^*)d\mu\\&=\int_X d^k\tr_{d^k}(\eta(\theta^{-k}(x))^*\xi(\theta^{-k}(x))d\mu=\int_X d^k\tr_{d^k}(\eta^*\xi(\theta^{-k}(x))d\mu\\&=\int_X\tr_{d^k}(\eta^*\xi(x))d\mu=\tau_\mu(\eta^*\xi).
\end{align*}
In the same way one can show that $\tau_\mu(b_k^*a_k)=\tau_\mu(a_kb_k^*)$ for all $a_k=\xi a$ and $b_k=\eta b$ with $\xi,\eta\in\E^k$ and $a,b\in\CP(\E)^0$. This is in fact enough to prove that $\tau_\mu(b_k^*a_k)=\tau_\mu(a_kb_k^*)$ holds for arbitrary $a_k,b_k\in\CP(\E)^k$. 

Now take $a,b\in\CP(\E)$. We can approximate $a$ and $b$ by finite linear combinations $\sum_{k\in\Z}a_k$ and $\sum_{k\in\Z}b_k$, respectively, where $a_k,b_k\in\CP(\E)^k$. Then calculate
\begin{align*}
\tau(\left(\sum_{k\in\Z}a_k\right)\left(\sum_{l\in\Z}b_l\right))&=\sum_{k,l\in\Z}\tau(a_kb_l)=\sum_{k,l\in\Z}\tau_\mu(E(a_kb_l))=\sum_{k\in\Z}\tau_\mu(a_kb_{-k})=\sum_{k\in\Z}\tau_\mu(b_{-k}a_k)\\&=\tau(\left(\sum_{l\in\Z}b_l\right)\left(\sum_{l\in\Z}a_l\right))
\end{align*}
which shows that $\tau(ab)$ equals $\tau(ba)$. 

To conclude the proof we show that the image of $\rho\circ\pi$ is contained in $M_\theta^1(X,d)$. Let $\tau$ be a trace on $\CP(\E)$. By restricting $\tau$ to $C_0(X)\subset\CP(\E)$ we obtain a measure $\mu\in M^1(X)$ such that $\tau|_{C_0(X)}=\tau_\mu$. We have to show that $\mu$ is $d$-conformal with respect to $\theta$. Take $f\in C_0(V)$ and $\varepsilon>0$. We can find elements $\{\xi_{ik}:i=1,...,m;\; k=1,...,d\}$ of $\E$ such that  
\[\|\left(\sum_{i=1}^m\sum_{k=1}^d\xi_{ik}\xi_{ik}^*\right)f-f\|<\varepsilon/2,\quad\text{and}\quad \|\left(\sum_{i=1}^m\xi_{ik}^*\xi_{ik}\right)(f\circ\theta)- f\circ\theta\|<\varepsilon/2\]
for all $k\in\{1,...,d\}$. We have 
    \begin{align*}
    \tau\left(f\sum_{i=1}^m\sum_{k=1}^d\xi_{ik}\xi_{ik}^*\right)&=\sum_{i=1}^m\sum_{k=1}^d\tau(f\xi_{ik}\xi_{ik}^*)\\&=\sum_{i=1}^m\sum_{k=1}^d\tau(\xi_{ik}(f\circ\theta)\xi_{ik}^*)=\sum_{i=1}^m\sum_{k=1}^d\tau(\xi_{ik}^*\xi_{ik}(f\circ\theta))
    \end{align*}
    and hence
    \begin{align*}
    |\tau(f)-d\tau(f\circ\theta)|\leq |\tau(f)-\tau\left(f\sum_{i=1}^m\sum_{k=1}^d\xi_{ik}\xi_{ik}^*\right)|+|\sum_{i=1}^m\sum_{k=1}^d\tau(\xi_{ik}^*\xi_{ik}(f\circ\theta))-d\tau(f\circ\theta)|<\varepsilon
    \end{align*}
    which shows that $\mu$ is $d$-conformal with respect to $\theta$.

	That $\rho\circ\pi$ is an affine homeomorphism from $T(\CP(\E))$ to $M_\theta^1(X,d)$, and that the trace associated to a conformal measure has the form given in the statement, now follows from Proposition \ref{prop:traces fp algebra} and Lemma \ref{lem:restricting traces}. 
\end{proof}

\begin{remark}
For the case $d=1$, meaning that $\V$ is a line bundle, Proposition \ref{prop:traces cp algebra} yields \cite[Proposition 4.23]{partI} as a special case. 
\end{remark}
\begin{remark}\label{rem:nonconstant rank}
The assumption in Proposition \ref{prop:traces cp algebra} that the vector bundle $\V$ has constant rank $d\in\N$ is actually unnecessary. In general the rank will only be constant on the connected components of $U$. To generalize the proposition, write $D$ for the function sending $x\in U$ to the rank $d(x)$. Then a $D$-conformal probability measure is a measure $\mu\in M^1(X)$ such that 
\[\mu(\theta(Y))=\int_YD(x)d\mu(x)\]
for all measurable $Y\subset U$. In Proposition \ref{prop:conformal measures} one has to replace the assumption $d>1$ by the existence of an $x\in X$ such that $D(x)>1$. The only change in the statement of the proposition is that instead of $\mu(D_{\gl})=0$, we have \[\mu(D_{\gl}\cap D^{-1}(k))=0\quad\text{for all }k>1.\]
However, $\mu(D_{\gl})=0$ holds for any $\mu\in M_\theta^1(X,D)$ if we additionally assume that $\theta$ is minimal.

 Proposition \ref{prop:traces cp algebra} remains true if we replace $d$-conformal with $D$-conformal probability measures. 
\end{remark}
\begin{remark}
If the vector bundle $\V$ is the trivial bundle of rank $d$, then $\CP(\Gamma_0(\V,\theta))$ is isomorphic to the $C^*$-algebra associated to a topological graph with vertex space $E^0=X$ and edge space $E^1=U\times\{1,...,d\}$. The range and source map are given by
\[r(x,k)=\theta(x)\quad\text{and}\quad s(x,k)=x,\]
respectively. A $d$-conformal measure with respect to $\theta$ is then the same as an invariant measure in the sense of \cite[Definition 4.1]{Schafhauser:2016}. In that case, Proposition \ref{prop:traces cp algebra} follows from the main result of \cite{Schafhauser:2016}.
\end{remark}

\section{Classifiability of the Cuntz--Pimsner algebra}\label{sect:classifiability of the Cuntz--Pimsner algebra}
In this section we show the main result of this paper, namely that the Cuntz--Pimsner algebras associated to partial automorphisms twisted by vector bundles are classifiable (see Section \ref{sect:classification program}) if the space is compact, infinite and has finite covering dimension, and the partial automorphism is minimal. 

\subsection{Decomposition rank for actions with finitely supported domains}\label{subsect: rshd of cp algebras}
Let $X$ be a locally compact second countable Hausdorff space, $\theta:U\to V$ be a partial automorphism on $X$, $\V$ be a vector bundle over $U$, and let $\E\coloneqq\Gamma_0(\V,\theta)$ be the associated $C^*$-correspondence. 

In this section we first show that if the orbit space $X/\Z$ of $\theta$ is locally compact and metrizable and the quotient map is proper, then the Cuntz--Pimsner algebra $\CP(\E)$ is a $C_0(X/\Z)$-algebra. We then verify that this is indeed the case if all orbits of the partial automorphism have the same length $N$. Using the $C_0(X/\Z)$-algebra structure in this case, we are able to prove that $\CP(\E)$ is subhomogeneous and has decomposition rank bounded by $\dim(X)$. This follows ideas from \cite{Hirshberg2017} in the case of global group actions with bounded orbits. Employing techniques from \cite{Geffen2021} we can extend the result to partial automorphisms with finitely supported domains, meaning that $D_N=\emptyset$ for some $N\in\N$. 

 Abusing notation we use $\orb(x)$ to denote both the orbit of $x$ as a subset of $X$ and the corresponding point in $X/\Z$. For the quotient map onto the orbit space we write $\pi:X\to X/\Z$, and $\pi^*$ denotes the induced map from $C_0(X/\Z)$ to $C_0(X)$. We require $\pi$ to be proper so that $\pi^*$ indeed maps into $C_0(X)$.

\begin{lemma}\label{lem:density functions on orbit space}
    Assume that the orbit space $X/\Z$ of $\theta$ is locally compact metrizable and that the quotient map $\pi:X\to X/\Z$ is proper. Then $\pi^*(C_0(X/\Z))C_0(X)$ is dense in $C_0(X)$. Also, $\pi^*(C_0(X/\Z\backslash\orb(x)))C_0(X)$ is dense in $C_0(X\backslash\orb(x))$ for every $x\in X$.
\end{lemma}
\begin{proof}
    Take a function $f$ in $C_0(X)$ and $\varepsilon>0$. There exists a compact subset $K$ of $X$ such that $|f|$ is bounded by $\varepsilon$ outside of $K$. The image $\pi(K)$ is a compact set in the orbit space $X/\Z$. Since $X/\Z$ is metrizable it is normal, and we can apply Tietze's extension theorem. We obtain a function $h$ in $C_0(X/\Z)$ which is $1$ on $\pi(K)$, and bounded by $1$ outside of $\pi(K)$. This means that $\pi^*(h)$ is $1$ on $K$, and bounded by $1$ outside of $K$. Therefore $||f-\pi^*(h)f||<\varepsilon$. This shows that $\pi^*(C_0(X/\Z))C_0(X)$ is dense in $C_0(X)$. The second statement can be viewed as a special case of the first one by replacing $X$ with $X\backslash\orb(x)$.
\end{proof}

We recall from \cite{partI} how to restrict the $C^*$-correspondence $\E\coloneqq\Gamma_0(\V,\theta)$ to subsets of $X$ invariant under the action.
\begin{definition}\label{def:restriction}[{\cite[Definition 4.10]{partI}}]
Let $W$ be a locally closed, $\theta$-invariant subset of $X$. Denote by $\theta_W:U\cap W\to V\cap W$ the partial automorphism on $W$ obtained by restricting $\theta$. We write $\E_W$ for the $C_0(W)$-correspondence $\Gamma_0(\V|_{W},\theta_W)$. 
\end{definition}

\begin{proposition}\label{prop:fibered structure on cp algebra}
    Assume that the orbit space $X/\Z$ of $\theta$ is locally compact metrizable and that the quotient map $\pi:X\to X/\Z$ is proper. Then the Cuntz--Pimsner algebra $\CP(\E)$ is a $C_0(X/\Z)$-algebra. The fiber at a point $\orb(x)\in X/\Z$ is given by $\CP(\E_{\orb(x)})$. 
\end{proposition}
\begin{proof}
    Functions in $\pi^*(C_0(X/\Z))$ are constant along $\theta$-orbits. Thus if $f$ lies in $\pi^*(C_0(X/\Z))$ then we have $f\xi=\xi f$ for all $\xi\in\E$. This implies $fa=af$ for all $a\in\CP(\E)$. Hence $f$ lies in the center of $\CP(\E)$, and therefore also in the center of the multiplier algebra of $\CP(\E)$.
    
    From Lemma \ref{lem:density functions on orbit space} we obtain that $\pi^*(C_0(X/\Z))\CP(\E)$ is dense in $\CP(\E)$, because $\CP(\E)=C_0(X)\CP(\E)$. In conclusion, $\pi^*$ defines a $C_0(X/\Z)$-structure on $\CP(\E)$.

    We want to compute the fiber at a point $\orb(x)\in X/\Z$, for some $x\in X$. By Lemma \ref{lem:density functions on orbit space} $\pi^*(C_0(X/\Z\backslash\orb(x)))C_0(X)$ is dense in $C_0(X\backslash\orb(x))$, and hence $\pi^*(C_0(X/\Z\backslash\orb(x)))\CP(\E)$ is equal to $C_0(X\backslash\orb(x))\CP(\E)$. Since $\orb(x)$ is closed and $\theta$-invariant, \cite[Proposition 4.11]{partI} yields the short exact sequence
\begin{align*}
    0\to\CP(\E_{X\backslash\orb(x)})\to\CP(\E)\to \CP(\E_{\orb(x)})\to 0.
\end{align*}
Additionally, we obtain from \cite[Proposition 4.11]{partI} that $C_0(X\backslash \orb(x))\CP(\E)$ is isomorphic to $\CP(\E_{X\backslash\orb(x)})$. Thus the fiber of $\CP(\E)$ at a point $\orb(x)$ is given by 
\begin{align*}
    \CP(\E)_{\orb(x)}&=\CP(\E)/(\pi^*(C_0(X/\Z\backslash\orb(x)))\CP(\E))=\CP(\E)/(C_0(X\backslash\orb(x))\CP(\E))\\&\cong\CP(\E)/\CP(\E_{X\backslash\orb(x)})\cong \CP(\E_{\orb(x)}).
\end{align*}
This concludes the proof.
\end{proof}
\begin{remark}
   That the quotient map onto the orbit space is proper was only assumed for convenience. If it is not, then we get an induced map $\pi^*$ from $C_0(X/\Z)$ into $C_b(X)$, the $C^*$-algebra of bounded continuous functions on $X$. This still defines a $C_0(X/\Z)$-structure on $\CP(\E)$.
\end{remark}

In the setting of global group actions, the orbit space is Hausdorff if the action has uniformly bounded orbits (see \cite[ Lemma 3.2]{Hirshberg2017}). However, for partial automorphisms even $D_N=\emptyset$ for some $N\in\N$ is not enough, as the following example shows:
\begin{example}
    Let $X=[0,1]\sqcup[2,3]$, $D_{-1}=(0,1]$, $D_1=(2,3]$, and $\theta:D_{-1}\to D_1$ defined as $\theta(x)=x+2$. Then $D_2=D_{-2}=\emptyset$. However, the points $\{0\}$ and $\{2\}$ cannot be separated by $\theta$-invariant neighborhoods. The same holds for their orbits, since $\orb(\{0\})=\{0\}$ and $\orb(\{2\})=\{2\}$. By the definition of the quotient topology, that means that the points $\orb(\{0\})$ and $\orb(\{2\})$ in $X/\Z$ cannot be separated by neighborhoods. Hence the orbit space is not Hausdorff.  
\end{example}

   Instead, we have to assume that $\theta$ has only orbits of one fixed size $N$. 

\begin{lemma}\label{lem: invariant neighborhood}
    Assume that all orbits of $\theta$ have length $N$ for some $N\in\N$, and none are periodic. Consider the orbit $\orb(x)$ of some point $x\in X$, and let $W\subset X$ be open with $\orb(x)\subset W$. There exists a $\theta$-invariant open neighborhood of $\orb(x)$ whose closure is compact and contained in $W$.
\end{lemma}
\begin{proof}
We can assume that $x$ lies in $D_{1-N}$. There exists an open neighborhood $A$ of $x$ whose closure is compact and contained in $D_{1-N}$, and such that $\theta^n(A)$ has compact closure which is contained in $W$ for all $n\in\{0,...,N-1\}$.  Since $A\subset D_{1-N}$, and $D_N=D_{-N}=\emptyset$, we have $A\cap D_1=\emptyset$ and $\theta^{N-1}(A)\cap D_{-1}=\emptyset$, whereas $\theta^l(A)\cap D_1=\theta^l(A)$ for all $l\in\{1,...,N-1\}$ and $\theta^l(A)\cap D_{-1}=\theta^l(A)$ for all $l\in\{0,...,N-2\}$. Hence 
\begin{align*}
    \theta(\bigcup_{l=0}^{N-1}\theta^l(A)\cap D_{-1})= \theta(\bigcup_{l=0}^{N-2}\theta^l(A))=\bigcup_{l=1}^{N-1}\theta^l(A)\subset \bigcup_{l=0}^{N-1}\theta^l(A)\quad\text{and}\\
     \theta^{-1}(\bigcup_{l=0}^{N-1}\theta^l(A)\cap D_{1})= \theta^{-1}(\bigcup_{l=1}^{N-1}\theta^l(A))=\bigcup_{l=0}^{N-2}\theta^l(A)\subset \bigcup_{l=0}^{N-1}\theta^l(A). 
\end{align*}
Thus $\bigcup_{l=0}^{N-1}\theta^l(A)$ is an open, $\theta$-invariant subset whose closure is compact and contained in $W$. 
\end{proof}

\begin{lemma}\label{lem:orbit space hausdorff}
Under the assumptions of Lemma \ref{lem: invariant neighborhood}, the orbit space $X/\Z$ is locally compact Hausdorff. 
\end{lemma}
\begin{proof}
Observe that if $W$ is a $\theta$-invariant subset of $X$, then $W=\pi^{-1}(\pi(W))$. Hence it follows from the definition of the quotient topology that if $W$ is open or closed, then $\pi(W)$ is as well. It is now straightforward to show, using Lemma \ref{lem: invariant neighborhood}, that the orbit space is locally compact Hausdorff.  
\end{proof}
\begin{lemma}\label{lem:quotient map proper}
     Under the assumptions of Lemma \ref{lem: invariant neighborhood}, the quotient map $\pi:X\to X/\Z$ is proper and open. This implies that the orbit space is metrizable.
\end{lemma}
\begin{proof}
    That the quotient map is proper can be shown in the same way as in the proof of  \cite[Lemma 3.2]{Hirshberg2017}, using Lemma \ref{lem: invariant neighborhood}. 

    For openness, let $W$ be an open subset of $X$. Then the set $\Tilde{W}\coloneqq\bigcup_{n=-N+1}^{N-1} \theta^n(W\cap D_{-n})$ is open, $\theta$-invariant and has the same image in the orbit space as $W$. Thus $\pi(W)=\pi(\Tilde{W})$ is open. 

    A continuous open image of a second countable space is second countable. By Lemma \ref{lem:orbit space hausdorff} the orbit space is locally compact Hausdorff, and hence regular. We obtain from Urysohn's metrization theorem that the orbit space is metrizable. 
\end{proof}

\begin{definition}
    A $C^*$-algebra is called \emph{$n$-homogeneous} or \emph{$n$-subhomogeneous} if all its irreducible representations have dimension $n$ or at most $n$, respectively. It is called \emph{homogeneous} or \emph{subhomogeneous} if there exists an $n$ such that it is $n$-homogeneous or $n$-subhomogeneous, respectively.
\end{definition}
\begin{lemma}\label{lem:finite space}
    Let $N>1$. If $X=\{x_1,...,x_N\}$ is finite and $\theta:\{x_1,...,x_{N-1}\}\to \{x_2,...,x_N\}$ maps $x_k$ to $x_{k+1}$ for all $k=1,...,N-1$, then $\CP(\Gamma(\V,\theta)))\cong M_l(\C)$ for some $l\in\N$. 
\end{lemma}
\begin{proof}
Since $X$ is finite, both $C(X)$ and $\Gamma(\V,\theta)$ are finite dimensional. Hence each spectral subspace $\CP(\Gamma(\V,\theta))^n$ is finite-dimensional. Clearly the $N$-th domain $D_N$ of $\theta$ is empty, and therefore multiplying $N$ or more sections in the Cuntz--Pimsner algebra always yields zero. In particular, all spectral subspaces $\CP(\Gamma(\V,\theta))^n$ for $n\geq N$ are trivial. Thus the Cuntz--Pimsner algebra is finite dimensional. The action is free and minimal, which by \cite[Corollary 4.17]{partI} implies that $\CP(\Gamma(\V,\theta)))$ is simple. Now \cite[Proposition 8.1.1]{Strung:2021} yields the statement. 
\end{proof}

\begin{proposition}\label{prop:periodic orbits of fixed length}
    If $X$ is compact and all orbits of $\theta$ have length $N$ for some $N\in\N$, and none are periodic, then $\CP(\Gamma_0(\V,\theta))$ is subhomogeneous and has decomposition rank bounded by $\dim(X)$.
\end{proposition}
\begin{proof}
    By Proposition \ref{prop:fibered structure on cp algebra}, $\CP(\E)$ is a $C(X/\Z)$-algebra with fibers $\CP(\E)_{\orb(x)}$ isomorphic to $\CP(\E_{\orb(x)})$. By Lemma \ref{lem:finite space}, the fibers are matrix algebras. Thus $\CP(\E)$ is subhomogeneous with primitive ideal space homeomorphic to $X/\Z$. The main result of \cite{Winter2004} together with \cite[Proposition 2.16, Chapter 9]{Pears1975} yields that its decomposition rank is bounded by $\dim(X)$. 
\end{proof}

\begin{theorem}\label{thm:finitely supported domains rsh}
If $X$ is a compact second countable Hausdorff space, $\theta:D_{-1}\to D_1$ is a partial automorphism on $X$ such that $D_{N}=D_{-N}=\emptyset$ for some $N\geq1$, and $\V$ is a vector bundle over $D_{-1}$, then the corresponding Cuntz--Pimsner algebra $\CP(\Gamma_0(\V,\theta))$ is subhomogeneous with decomposition rank at most $\dim(X)$. 
\end{theorem}
\begin{proof}
The proof is inspired by the proof of \cite[Theorem 4.4]{Geffen2021}. As usual, we write $\E\coloneqq\Gamma_0(\V,\theta)$. 

Define $X^{(k)}\coloneqq\{x\in X:|\orb(x)|=k\}$ and $X_k\coloneqq \bigcup_{l=1}^k X^{(l)}$ for $1\leq k\leq N$. The set $X_k$ consists of all points contained in orbits of length less or equal then $k$. Correspondingly a point $x\in X$ lies in the complement of $X_k$ if and only if it is contained in an orbit of length strictly greater than $k$, which is the case if and only if there exist $l,j\geq 0$ such that $l+j=k$ and $x$ lies in both $D_l$ and $D_{-j}$. Take $x\in X$ and $l,j\geq 0$ such that this holds. Then there is an open neighborhood $W$ of $x$ which is contained in both $D_l$ and $D_{-j}$, which implies that $W$ is contained in the complement of $X_k$. Hence $X_k$ is closed,  and therefore $X^{(k)}=X_k\backslash X_{k-1}$ is relatively open in $X_k$ for all $k\in\{1,...,N\}$. Furthermore it is clear that each $X^{(k)}$ is $\theta$-invariant. Thus \cite[Proposition 4.11]{partI} yields the short exact sequence 
\begin{align}\label{eq:short exact sequence}
0\to\CP(\E_{X^{(k)}})\to\CP(\E_{X_k})\to\CP(\E_{X_{k-1}})\to 0
\end{align}
for every $k\in\{2,...,N\}$. 

By Proposition \ref{prop:periodic orbits of fixed length} $\CP(\E_{X^{(k)}})$ is homogeneous and has decomposition rank bounded by $\dim(X)$. 

To finish the argument, notice that for $k=2$ the short exact sequence from Equation \ref{eq:short exact sequence} takes the form
\begin{align*}
0\to\CP(\E_{X^{(2)}})\to\CP(\E_{X_2})\to\CP(\E_{X^{(1)}})\to 0,
\end{align*}
and since the restriction of $\theta$ to $X^{(1)}=X\backslash(D_{-1}\cup D_1)$ has empty domains, we obtain $\CP(\E_{X^{(1)}})=C_0(X^{(1)})$. It is clear that $C_0(X^{(1)})$ is 1-homogeneous and by \cite[Theorem 17.2.4]{Strung:2021} has decomposition rank $\dr(C_0(X^{(1)}))\leq\dim(X^{(1)})\leq\dim(X)$, where for the last inequality we used Proposition 6.4 in Chapter 3 of \cite{Pears1975}. The statement now follows from repeatedly applying \cite[Lemma 4.3]{Geffen2021} to the short exact sequences from Equation (\ref{eq:short exact sequence}).
\end{proof}
\begin{corollary}\label{cor:decomposition rank}
Let the assumptions be the same as in Theorem \ref{thm:finitely supported domains rsh}, except that $X$ is only assumed to be locally compact instead of compact. Then $\CP(\Gamma_0(\V,\theta))$ is subhomogeneous with decomposition rank at most $\dim(X)$.
\end{corollary}
\begin{proof}
We can use the same argument as in the proof of \cite[Corollary 4.6]{Geffen2021}: The one-point compactification $X^+$ of $X$ is a compact metrizable space. Define a partial automorphism $\theta^+$ on $X^+$ by setting $D_{-1}^+\coloneqq D_{-1}$, $D_{1}^+\coloneqq D_{1}$ and $\theta^+(x)=\theta(x)$ for all $x\in D_{-1}$. Since the domains of $\theta$ and $\theta^+$ are the same, we have $D_N^+=\emptyset$. Hence Theorem \ref{thm:finitely supported domains rsh} applies, and $\CP(\Gamma_0(\V,\theta^+))$ is subhomogeneous and has decomposition rank bounded by $\dim(X^+)=\dim(X)$. Since $X$ is an open, $\theta^+$-invariant subset of $X^+$, the Cuntz--Pimsner algebra $\CP(\Gamma_0(\V,\theta))$ is an ideal of $\CP(\Gamma_0(\V,\theta))$ by \cite[Proposition 4.11]{partI}, and hence it is subhomogeneous and has decomposition rank at most $\dim(X)$ as well. 
\end{proof}

\subsection{Nuclear dimension and classifiability}
The following theorem is the key step in showing that our Cuntz--Pimsner algebras are classifiable. 
\begin{theorem}\label{thm:bound on nuclear dimension}
    Let $\theta:U\to V$ be a partial automorphism on a locally compact second countable Hausdorff space $X$, $\V$ be a vector bundle over $U$, and $\E\coloneqq\Gamma_0(\V,\theta)$. Then
    \begin{align*}
        \dimnuc(\CP(\E))\leq\dimnuc(\CP(\E_{\Tilde{X}}))+2\dim(X)+2,
    \end{align*}
    where $\Tilde{X}\coloneqq\bigcap_{n\in\Z}\overline{D_n}$, and $\E_{\tilde{X}}$ is as in Definition \ref{def:restriction}. 
\end{theorem}
\begin{proof}

    The proof can be copied almost word by word from the proof of \cite[Theorem 6.2]{Geffen2021}. The set 
    \[Y_-\coloneqq\bigcap_{n\geq0}\overline{D_{-n}}\]
    is closed and $\theta$-invariant. By \cite[Proposition 4.11]{partI} we obtain the short exact sequence
    \[0\to \CP(\E_{X\backslash Y_-})\to \CP(\E)\to\CP(\E_{Y_-})\to 0.\]
As shown in \cite{Geffen2021}, the partial automorphism $\theta|_{X\backslash Y_-}$ can be obtained as an inductive limit of partial automorphisms $\theta^{(n)}$ with finitely supported domains. More precisely, there exists an increasing sequence $\{W_n\}_{n\geq 0}$ of open subsets of $X\backslash Y_-$ such that the union over all $W_n$ equals the domain of $\theta|_{X\backslash Y_-}$ , and such that the partial automorphism $\theta^{(n)}$ obtained by restricting the domain of $\theta|_{X\backslash Y_-}$ to $W_n$ has finitely supported domains, for all $n\in\N$. By \cite[Proposition 4.9]{partI} we have 
\[\CP(\E_{X\backslash Y_-})=\lim\limits_{\to}\CP(\Gamma_0(\V,\theta^{(n)})).\]
Due to the permanence properties of nuclear dimension together with Corollary \ref{cor:decomposition rank} we obtain 
\begin{align}\label{eq:estimate}\begin{split}\dim_{\mathrm{nuc}}(\CP(\E))&\leq\dim_{\mathrm{nuc}}(\CP(\E_{Y_-}))+ \lim\inf_n\dim_{\mathrm{nuc}}(\CP(\Gamma_0(\V,\theta^{(n)})))+1\\&\leq \dim_{\mathrm{nuc}}(\CP(\E_{Y_-}))+\dim(X)+1.\end{split}\end{align}

To bound the nuclear dimension of $\CP(\E_{Y_-})$ one can essentially run the same argument: The proof of \cite[Theorem 6.2]{Geffen2021} shows that
\[\Tilde{X}\coloneqq\bigcap_{n\in\Z}\overline{D_n}\]
is a closed $\theta$-invariant subset of $Y_-$, and that $\theta|_{Y_-\backslash \tilde{X}}$ is the inductive limit of partial automorphisms with finitely supported domains. Together with Equation (\ref{eq:estimate}), this establishes the estimate in the statement of the theorem. 
\end{proof}
\begin{corollary}\label{cor:nuclear dimension}
    If $\bigcap_{n\in\Z}\overline{D_n}=\emptyset$ then
    \begin{align*}
        \dimnuc(\CP(\E))\leq2\dim(X)+2.
    \end{align*}
\end{corollary}
\begin{theorem}\label{thm:finite nuclear dimension}
    If in addition to the assumptions of Theorem \ref{thm:bound on nuclear dimension} we require $\theta$ to be minimal and $U\subsetneq X$, then the nuclear dimension of $\CP(\E)$ is finite. 
\end{theorem}
\begin{proof}

  Let $\{U_k\}_{k\geq 0}$ be a sequence of open sets exhausting $U$. Define $\beta^{(k)}$ to be the partial automorphism on $X$ whose domain $D_{-1}^{\beta^{(k)}}$ is $U_k$, and $\beta^{(k)}(x)=\theta(x)$ for all $x\in U_k$. Write $K_n^{(k)}$ for the $n$-th domain of $\beta^{(k)}$. Note that $K_{-1}^{(k)}=U_k$. We have 
  \[\CP(\E)=\lim\limits_{\to}\CP(\Gamma_0(\V,\beta^{(k)}))\]
  by \cite[Proposition 4.9]{partI}. 
  
   It is shown in the proof of \cite[Theorem 7.5]{Geffen2021} that we can choose the sequence $\{U_k\}_{k\geq 0}$ such that 
  \[ \bigcap_{n\geq0}\overline{K_n^{(k)}}=\emptyset \]
  for all $k\geq 0$. We obtain 
  \[\dim_{\mathrm{nuc}}(\CP(\E))\leq\lim\inf_k\dim_{\mathrm{nuc}}(\CP(\Gamma_0(\V,\beta^{(k)})))\leq 2\dim(X)+2\]
  from Corollary \ref{cor:nuclear dimension}.
\end{proof}

Recall from Section \ref{sect:orbit types} that $D_-$ denotes the set
\[\left(\bigcap_{n>0}D_n\right)\backslash \left(\bigcap_{n\in\Z}D_n\right).\]
\begin{theorem}\label{thm:classifiable}
    If $\theta:U\to V$ is minimal with $U\subsetneq X$ and $X$ is an infinite compact second countable Hausdorff space with finite covering dimension, then the nuclear dimension of $\CP(\E)$ is at most one. This implies that $\CP(\E)$ is classifiable. It is stably finite if the vector bundle is a line bundle or if $D_-$ is nonempty. Otherwise $\CP(\E)$ is purely infinite. 
\end{theorem}
\begin{proof}
   If the space is infinite then any minimal action is automatically free. Hence $\CP(\E)$ is simple by \cite[Corollary 4.17]{partI}. Compactness of $X$ implies that $\CP(\E)$ is unital. From Theorem \ref{thm:finite nuclear dimension} we obtain that $\CP(\E)$ has finite nuclear dimension. Since $X$ is metrizable, $C(X)$ is separable. Together with the fact that $\E$ is countably generated, this implies that $\CP(\E)$ is separable. Proposition 8.8 from \cite{Katsura:2004} shows that $\CP(\E)$ satisfies the UCT. 

   Theorems A and B from \cite{Castillejos2021} yield that the Cuntz--Pimsner algebra is classifiable and has nuclear dimension at most one. Since it is simple, tracial states are automatically faithful. It then follows from  \cite[Proposition 4.22]{partI} that $\CP(\E)$ has a faithful tracial state if $\V$ is a line bundle. Combining Propositions \ref{prop:conformal measures} and \ref{prop:traces cp algebra} (or Remark \ref{rem:nonconstant rank} if the vector bundle does not have constant rank) yields that if $\V$ is not a line bundle, then it is stably finite if and only if $D_-$ is nonempty. This concludes the proof. 
\end{proof}

\begin{remark}
Note that Theorem \ref{thm:classifiable} assumes $U\subsetneq X$. The case that $U$ equals $X$, that is, the case of a global action, was handled in \cite{adamo2023}. 
\end{remark}

\section{Orbit-breaking subalgebras}\label{sect:Orbit breaking subalgebras}

In this section, we apply the main results from the previous sections to orbit-breaking subalgebras. For an overview of the history of orbit-breaking subalgebras, see \cite{adamo2023}. We are mainly interested in generalizing results from \cite{phillips:2007} and \cite{adamo2023}. First, we need prove some results about restricting partial automorphisms.

\subsection{Restricting partial automorphisms}

Let $W$ be an open subset of $X$. We recall the definition of the restricted partial automorphism $\theta^{(W)}$ on $X$ from \cite{partI}. 
\begin{definition}[{\cite[Definition 4.7]{partI}}]
    For a partial automorphism $\theta:U\to V$ and an open subset $W\subset U$ we define a partial automorphism $\theta^{(W)}:W\to\theta(W)$ on $X$ by $\theta^{(W)}(x)=x$ for all $x\in W$. The $n$-th domain of $\theta^{(W)}$ is denoted by $D_n^{(W)}$. 
\end{definition}

\begin{remark}
    In the following, we will write $\orb(x)$ for the $\theta$-orbit of a point $x$, and $\orb^{W}(x)$ for the $\theta^{(W)}$-orbit. 
\end{remark}
\begin{definition}\label{def:restricted partial auto}
    Let $Y$ be a closed subset of $U$. 
    \begin{enumerate}
        \item We say that \emph{taking out $Y$ preserves orbit size} if $|\orb^{U\backslash Y}(x)|=|\orb(x)|$ for all $x\in X$.
        \item We say that \emph{taking out $Y$ preserves dense orbits} if whenever the $\theta$-orbit of a point $x\in X$ is dense in $X$, the $\theta^{(U\backslash Y)}$-orbit of $x$ is dense as well.
        \item We say that $Y$ \emph{meets every $\theta$-orbit at most once} if $\theta^n(y)\cap Y\subset\{y\}$ for all $n\in\Z\backslash\{0\}$ and $y\in Y\cap D_{-n}$. 
    \end{enumerate}
\end{definition}

Given a partial automorphism $\theta:U\to V$ and a closed subset $Y$ of $U$, we want to investigate under which conditions the restricted partial automorphism $\theta^{(U\backslash Y)}$ has the same set of invariant measures as $\theta$. It turns out that periodic orbits of $\theta$ need to be treated differently from the rest of the space. Let therefore $X^{(l)}$ be the set of all points of $X$ which have a periodic orbit of size $l$, namely $X^{(l)}=\{x\in X:|\orb(x)|=l\}\cap D_{\gl}$. Let $X_k$ be the set of all points of $X$ with a periodic orbit of size at most $k$, namely $X_k\coloneqq \bigcup_{l=1}^kX^{(l)}$. We allow $k=\infty$ to denote the set of all points with a periodic orbit of any size. Note that the sets $X^{(k)}$ and $X_k$ are very similar, but not identical, to the eponymous sets in the proof of Theorem \ref{thm:finitely supported domains rsh}
\begin{lemma}\label{lem: points with bounded orbit closed}
    For every $k\in\N$ the set $X_k$ is closed in $X$.
\end{lemma}
\begin{proof}
    Let $x\in X\backslash X_k$. Then there exists an open neighborhood $W$ of $x$ such that $\theta^l(W)$ is disjoint from $\theta^j(W)$ for all $0\leq l<j\leq k$. This implies that $W$ is contained in $X\backslash X_k$, and hence $X_k$ is closed. 
\end{proof}
Lemma \ref{lem: points with bounded orbit closed} implies that $X^{(l)}=X_l\backslash X_{l-1}$ is measurable.
\begin{lemma}\label{lem:invariant measures periodic case}
    Let $L$ be a measurable subset of $X^{(l)}$ for some $l\in\N$. Let $Y\subset X$ be such that $Y$ meets every $\theta$-orbit at most once. Then $\mu(\theta(L))=\mu(L)$ for all $\mu\in M_{\theta^{(U\backslash Y)}}^1(X)$. 
\end{lemma}
\begin{proof}
    If $L$ is contained in $X^{(l)}$ then we have $\theta(x)=\theta^{(1-l)}(x)$ for every $x\in L$. Now take $x$ in $L\cap Y$. Since $Y$ meets every $\theta$-orbit at most once $x$ lies in $D_{l-1}^{(U\backslash Y)}$ and we have $\theta^{(1-l)}(x)=(\theta^{(U\backslash Y)})^{(1-l)}(x)$. This implies \[\theta(L\cap Y)=\theta^{(1-l)}(L\cap Y)=(\theta^{(U\backslash Y)})^{(1-l)}(L\cap Y),\]and so for every $\mu\in M_{\theta^{(U\backslash Y)}}^1(X)$ we obtain 
    \begin{align*}
        \mu(\theta(L))&=\mu(\theta(L\cap Y))+\mu(\theta(L\backslash Y))=\mu((\theta^{(U\backslash Y)})^{(1-l)}(L\cap Y))+\mu(\theta^{(U\backslash Y)}(L\backslash Y))\\
        &=\mu(L\cap Y)+\mu(L\backslash Y)=\mu(L).
    \end{align*}
    This concludes the proof.
\end{proof}
\begin{lemma}\label{lem: blurbsblurbs implies measure zero}
    Let $Y\subset D_{\gl}$ be such that $Y$ meets every $\theta$-orbit at most once. Let $L$ be a measurable subset of $Y$ such that no point in $L$ has a periodic orbit.  then $\mu(L)=0$ and $\mu(\theta(L))=0$ for every $\mu\in M_{\theta^{(U\backslash Y)}}^1(X)$. 
\end{lemma}
\begin{proof}
    It is easy to see that, under the above assumptions, $L$ is a subset of $D_-^{(U\backslash Y)}$, and $\theta(L)$ is a subset of $D_+^{(U\backslash Y)}$. Then use Proposition \ref{prop: no measures}.
\end{proof}
\begin{lemma}\label{lem:invariant measures nonperiodic case}
    Take $Y\subset D_{\gl}$ such that $Y$ meets every $\theta$-orbit at most once. Let $L$ be a measurable subset of $X$ such that no point in $L$ has a periodic orbit. Then $\mu(\theta(L))=\mu(L)$ for all $\mu\in M_{\theta^{(U\backslash Y)}}^1(X)$. 
\end{lemma}
\begin{proof}
    For every $\mu\in M_{\theta^{(U\backslash Y)}}^1(X)$ and every measurable $L\subset U$ we have 
    \begin{align*}
        \mu(\theta(L))&=\mu(\theta(L\cap Y))+\mu(\theta(L\backslash Y)).
    \end{align*}
    If no point in $L$ has a periodic orbit then according to Lemma \ref{lem: blurbsblurbs implies measure zero} we have $\mu(\theta(L\cap Y))=0$ and \[\mu(\theta(L\backslash Y))=\mu(\theta^{(U\backslash Y)}(L\backslash Y))=\mu(L\backslash Y)=\mu(L\backslash Y)+\mu(L\cap Y)=\mu(L),\]
    which shows $\mu(\theta(L))=\mu(L)$.
\end{proof}
\begin{proposition}\label{prop:blurbs}
    Let $Y\subset U$ be a closed set. Then the following are equivalent:
    \begin{enumerate}
        \item Taking out $Y$ preserves orbit size. 
        \item  The set $Y$ is contained in $D_{\gl}$ and meets every $\theta$-orbit at most once.
        \item Taking out $Y$ does not change the space of invariant measures, namely $M_{\theta^{(U\backslash Y)}}^1(X)=M_\theta^1(X)$.
    \end{enumerate}
\end{proposition}
\begin{proof}
    Assume that taking out $Y$ does not preserve orbit size, that is there exists a point $x$ in $X$ such that \[|\orb^{U\backslash Y}(x)|<|\orb(x)|.\] This necessarily means that $\orb^{U\backslash Y}(x)$ is finite. Hence $M_{\theta^{(U\backslash Y)}}^1(X)$ contains the atomic measure associated to this finite set, which is not contained in $M_{\theta}^1(X)$. Therefore (3) implies (1).
    
    Assume that (2) does not hold. If there is $y\in Y$ which does not lie in $D_{\gl}$, then either $\orb_-(y)$ or $\orb_+(y)$ must be finite. The $\theta^{(U\backslash Y)}$-orbit $\orb^{(U\backslash Y)}(y)$ is contained in $\orb_-(y)$.  Since $y$ lies in $U$ the backward orbit $\orb_-(y)$ is a proper subset of $\orb(y)$. Hence if $\orb_-(y)$ is finite then we have $|\orb^{(U\backslash Y)}(y)|\leq|\orb_-(y)|<|\orb(y)|$ which means that taking out $Y$ does not preserve orbit size. The case that $\orb_+(y)$ is finite can be dealt with similarly. Thus if $Y$ is not contained in $D_{\gl}$ then (1) does not hold. 

    On the other hand, assume that there exist $n\in\Z$ and $y_1\in Y\cap D_{-n}$ such that $\theta^n(y_1)\in Y$ and $\theta^n(y_1)\neq y_1$. Write $y_2\coloneqq\theta^n(y_1)$. We can assume without loss of generality that $n$ is positive, and that $\theta^k(y_1)\notin Y$ for all $1\leq k< n$. Then the $\theta^{(U\backslash Y)}$-orbit of $\theta(y_1)$ is equal to $\{\theta(y_1),...,y_2\}$. By assumption $y_2$ is the only element of $Y$ lying in $\orb^{(U\backslash Y)}(\theta(y_1))$, in particular $y_1\notin\orb^{(U\backslash Y)}(\theta(y_1))$. Hence 
    \[|\orb^{(U\backslash Y)}(\theta(y_1))|=|\{\theta(y_1),...,y_2\}|<|\{y_1,\theta(y_1),...,y_2\}|\leq|\orb(\theta(y_1))|,\] and again this shows that taking out $Y$ does not preserve orbit size. We have shown that (1) implies (2). 

    Now assume that (2) holds. It is clear that $M_{\theta}^1(X)$ is contained in $M_{\theta^{(U\backslash Y)}}^1(X)$. To show the other inclusion, note that we can decompose the space into points with non-periodic orbits and points with periodic orbits of a fixed size, namely
    \[X=(X\backslash X_\infty)\;\dot{\cup}\; X^{(1)}\;\dot{\cup}\; X^{(2)}\;\dot{\cup}\; ...\]
    where each of the sets in the decomposition is measurable and $\theta$-invariant. Given $\mu\in M_{\theta^{(U\backslash Y)}}^1(X)$ we can restrict $\mu$ to each of the sets in the decomposition. It follows from Lemmas \ref{lem:invariant measures periodic case} and \ref{lem:invariant measures nonperiodic case} that the restrictions of $\mu$ are $\theta$-invariant. Since the sets in the decomposition cover the whole space, this shows that $\mu$ is $\theta$-invariant, and hence $M_{\theta^{(U\backslash Y)}}^1(X)$ is contained in $M_{\theta}^1(X)$. Therefore \[M_{\theta^{(U\backslash Y)}}^1(X)=M_{\theta}^1(X)\] and hence (2) implies (3). This concludes the proof.
\end{proof}
\begin{proposition}\label{thm:blabla}
    Assume that the space $X$ is infinite. Let $\theta:U\to V$ be minimal, and $Y\subset U$ be a closed set. Then the following are equivalent:
     \begin{enumerate}
     \item The restricted partial automorphism $\theta^{(U\backslash Y)}$ is minimal.
         \item Taking out $Y$ preserves dense orbits.
         \item Taking out $Y$ preserves orbit size, and $\orb_+(y)$ as well as $\orb_-(y)$ are dense in $X$ for every $y\in Y$.
     \end{enumerate}
     If $X$ is compact then (1), (2) and (3) are furthermore equivalent to:
     \begin{enumerate}
         \item[(4)] Taking out $Y$ preserves orbit size.
          \item[(5)] The set $Y$ is contained in $D_{\gl}$ and meets every $\theta$-orbit at most once. 
          \item[(6)] We have $M_{\theta^{(U\backslash Y)}}^1(X)=M_\theta^1(X)$.
     \end{enumerate}
\end{proposition}
\begin{proof}
    That (1) and (2) are equivalent is clear from the characterization of minimality in Theorem \ref{thm:minimality partial case}.
    
    Now we show that (3) implies (2). Assume that (3) holds and take $x\in X$. By Proposition \ref{prop:blurbs} we either have $\orb(x)\cap Y=\emptyset$, or $\orb(x)\cap Y=\{y\}$ for some $y\in Y$. In the first case we are done, in the latter case assume without loss of generality that $y\in\orb_+(x)$. Then $\orb^{(U\backslash Y)}(x)=\orb_-(y)$ and $\orb_-(y)$ is by assumption dense in $X$.

    We show that (2) implies (3). For every $y\in Y$ we have that $\orb^{(U\backslash Y)}(y)$ is contained in $\orb_-(y)$, and $\orb^{(U\backslash Y)}(\theta(y))$ is contained in $\orb_+(y)$. Hence if $\orb_-(y)$ is not dense then neither is $\orb^{(U\backslash Y)}(y)$, and the same holds for $\orb_+(y)$ and $\orb^{(U\backslash Y)}(\theta(y))$. In both cases restricting to $Y$ does not preserve dense orbits. 
    
    If taking out $Y$ does not preserve orbit size then we obtain a point whose $\theta^{(U\backslash Y)}$-orbit is finite and strictly smaller then its $\theta$-orbit. If $X$ is infinite then a finite orbit cannot be dense, which shows that taking out $Y$ does not preserve dense orbits. If $X$ is finite and $y\in Y$, then the $\theta^{(U\backslash Y)}$-orbit of $\theta(y)$ does not contain $y$, which means that it is not dense. Note that we have used that there are no periodic orbits since $\theta$ is minimal. This shows that (2) implies (3).
    
   If $X$ is compact then Theorem \ref{thm:minimality partial case} implies that (3) is equivalent to (4). The other equivalences follow from Proposition \ref{prop:blurbs}.
\end{proof}
\subsection{Orbit-breaking subalgebras}
In the following, let $X$ be a locally compact Hausdorff space, $\V$ a vector bundle over $X$, $h:X\to X$ a homeomorphism, and $\E=\Gamma(\V,h)$. Let $Y$ be a closed subset of $X$, and define the partial automorphism $\theta:X\backslash Y\to h(X\backslash Y)$ on $X$ as the restriction of $h$. The algebra $\CP(\E C_0(X\backslash Y))$, meaning the Cuntz--Pimsner algebra associated to the $C^*$-correspondence $\E C_0(X\backslash Y)$ over $C_0(X)$, is called the \emph{orbit-breaking subalgebra of $\CP(\E)$ at $Y$}. It is isomorphic to the Cuntz--Pimsner algebra $\CP(\Gamma_0(\V\vert_{X\backslash Y},\theta))$. 

    The following two propositions, which concern invariant measures and minimality of orbit-breaking subalgebras, generalize Propositions 2.5 and 2.6 from \cite{Lin2010} as well as Corollary 6.15 from \cite{adamo2023}. Recall from Definition \ref{def:restricted partial auto} that $Y$ is said to meet every $h$-orbit at most once if $h^n(y)\cap Y\subset\{y\}$ for all $n\in\Z\backslash\{0\}$ and $y\in Y$. If $h$ is free, then this condition becomes $h^n(Y)\cap Y=\emptyset$ for all $n\in\Z\backslash\{0\}$, which is how it was stated in \cite{adamo2023}. 

Assume that $X$ is compact, $\theta$ is free, and that $\V$ is a line bundle. Then there is an affine continuous injective map from $T(\CP(\E))$ to $T(\CP(\E C_0(X\backslash Y)))$ obtained via the sequence 
\begin{align}\label{eq:restriction map}
    T(\CP(\E))\to M^1_{h}(X)\hookrightarrow M^1_{\theta}(X)\to  T(\CP(\E C_0(X\backslash Y))),
\end{align}
where the first and the last arrow are the affine homeomorphisms from \cite[Proposition 4.23]{partI}, and the arrow in the middle is the inclusion map. 
\begin{proposition}\label{prop:orbit breaking preserves traces}
    Assume that $X$ is compact, $\theta$ is free, and that $\V$ is a line bundle. Then the affine continuous map from $T(\CP(\E))$ to $T(\CP(\E C_0(X\backslash Y)))$ obtained by the sequence (\ref{eq:restriction map}) is a homeomorphism if and only if $Y$ meets every $h$-orbit at most once. 
\end{proposition}
\begin{proof}
    The third and the first map of (\ref{eq:restriction map}) are homeomorphisms by \cite[Proposition 4.23]{partI}. The assertion now follows from Proposition \ref{prop:blurbs}.
\end{proof}

\begin{proposition}\label{prop:orbit-breaking simplicity}
     Assume that $X$ is infinite and $h$ is minimal. Then the orbit-breaking subalgebra $\CP(\E C_0(X\backslash Y))$ is simple if and only if $Y$ meets every $h$-orbit at most once. 
\end{proposition}
\begin{proof}
    Since $h$ is minimal and $X$ is infinite, $h$ is free. Then $\theta$ is free as well, because fixed points of $\theta$ are also fixed points of $h$. Hence by \cite[Corollary 4.17]{partI}, the Cuntz--Pimsner algebra $\CP(\E C_0(X\backslash Y))$ is simple if and only if the partial automorphism $\theta$ is minimal. On the other hand, Proposition \ref{thm:blabla} yields that $\theta$ is minimal if and only if $Y$ meets every $h$-orbit at most once. This finishes the proof.
\end{proof}

In the next proposition we will use the notion of a \emph{recursive subhomogeneous} $C^*$-algebra, see \cite{phillips:2007}. The proposition generalizes \cite[Theorem 2.4]{Lin2010} as well as \cite[Theorem 8.8]{ForoughJeongStrung:2025}, except for the additional assumption $\dim(X)<\infty$. This additional assumption is the price we pay for not explicitly constructing the recursive subhomogeneous decomposition.
\begin{proposition}
        Assume that $Y$ has nonempty interior, that $X$ is compact, second countable and has finite covering dimension, and that $h$ is minimal. Then the Cuntz--Pimsner algebra $\CP(\E C_0(X\backslash Y))$ is recursive subhomogeneous with decomposition rank bounded by $\dim(X)$.
\end{proposition}
\begin{proof}
     There is a maximal return time for points in $Y$ (see Theorem 2.3 in \cite{Lin2010}). We obtain $D_N=\emptyset$ for the domains of $\theta$, where $N\in\N$ is chosen big enough. Now combine Theorem \ref{thm:finitely supported domains rsh} with the main result of \cite{Winter2004}.
\end{proof}

Finally we arrive at a generalization of Theorem 6.18 in \cite{adamo2023}:
\begin{theorem}\label{thm:orbit breaking classifiable}
    Assume that $X$ is infinite, compact and second countable, and that $h$ is minimal. If $Y$ meets every $h$-orbit at most once and $X$ has finite covering dimension, then the orbit-breaking subalgebra $\CP(\E C_0(X\backslash Y))$ is classifiable and stably finite.
\end{theorem}
\begin{proof}
    For a partial automorphism obtained from breaking the orbits of a global homeomorphim, the set $D_-$ is always nonempty. The statement now follows from combining Theorem \ref{thm:classifiable} and Proposition \ref{thm:blabla}.
\end{proof}
\begin{remark}
If $\V$ is not a line bundle, then the Cuntz--Pimsner algebra $\CP(\Gamma(\V,h))$ is purely infinite by \cite[Corollary 5.6]{adamo2023}. It is therefore quite remarkable that the orbit-breaking subalgebra $\CP(\Gamma_0(\V\vert_{X\backslash Y},\theta))$ in Theorem \ref{thm:orbit breaking classifiable} is stably finite. It means that $\CP(\Gamma_0(\V\vert_{X\backslash Y},\theta))$ cannot be a large subalgebra inside $\CP(\Gamma(\V,h))$ in the sense of Phillips \cite{phillips:2014}. It also implies that Proposition \ref{prop:orbit breaking preserves traces} cannot be generalized to higher-rank bundles. If $\V$ is not a line bundle, then breaking the orbit will create new tracial states on the Cuntz--Pimsner algebra. 
\end{remark}

\printbibliography

@misc{partI,
      title={Cuntz--Pimsner algebras of partial automorphisms twisted by vector bundles I: Fixed point algebra, simplicity and the tracial state space}, 
      author={Aaron Kettner},
      year={2024},
      eprint={2408.10047},
      archivePrefix={arXiv},
      primaryClass={math.OA},
      url={https://arxiv.org/abs/2408.10047}, 
}

@article {adamo2023,
    AUTHOR = {Adamo, Maria Stella and Archey, Dawn E. and Forough, Marzieh
              and Georgescu, Magdalena C. and Jeong, Ja A. and Strung, Karen
              R. and Viola, Maria Grazia},
     TITLE = {{${C}^*$}-algebras associated to homeomorphisms twisted by
              vector bundles over finite dimensional spaces},
   JOURNAL = {Trans. Amer. Math. Soc.},
  FJOURNAL = {Transactions of the American Mathematical Society},
    VOLUME = {377},
      YEAR = {2024},
    NUMBER = {3},
     PAGES = {1597--1640},
      ISSN = {0002-9947,1088-6850},
   MRCLASS = {46L35 (37A55 46H25 46L85)},
  MRNUMBER = {4744737},
       DOI = {10.1090/tran/8900},
       URL = {https://doi.org/10.1090/tran/8900},
}

@misc{BBSTWW:2Col,
      title={Covering dimension of $C^*$-algebras and 2-coloured classification}, 
      author={Joan Bosa and Nathanial P. Brown and Yasuhiko Sato and Aaron Tikuisis and Stuart White and Wilhelm Winter},
      year={2016},
      eprint={1506.03974},
      archivePrefix={arXiv},
      primaryClass={math.OA},
      url={https://arxiv.org/abs/1506.03974}, 
}

@article{CETWW,
   title={Nuclear dimension of simple $\mathrm {C}^*$-algebras},
   volume={224},
   ISSN={1432-1297},
   url={http://dx.doi.org/10.1007/s00222-020-01013-1},
   DOI={10.1007/s00222-020-01013-1},
   number={1},
   journal={Inventiones mathematicae},
   publisher={Springer Science and Business Media LLC},
   author={Castillejos, Jorge and Evington, Samuel and Tikuisis, Aaron and White, Stuart and Winter, Wilhelm},
   year={2020},
   month=dec, pages={245–290} }

@Article{Castillejos2021,
 Author = {Castillejos, Jorge and Evington, Samuel and Tikuisis, Aaron and White, Stuart and Winter, Wilhelm},
 Title = {Nuclear dimension of simple {{\(\mathrm{C}^*\)}}-algebras},
 FJournal = {Inventiones Mathematicae},
 Journal = {Invent. Math.},
 ISSN = {0020-9910},
 Volume = {224},
 Number = {1},
 Pages = {245--290},
 Year = {2021},
 Language = {English},
 DOI = {10.1007/s00222-020-01013-1},
 zbMATH = {7330744},
 Zbl = {1467.46055}
}

@Article{deeleyputnamstrung:2018jiangsu,
 Author = {Deeley, Robin J. and Putnam, Ian F. and Strung, Karen R.},
 Title = {Constructing minimal homeomorphisms on point-like spaces and a dynamical presentation of the {Jiang}-{Su} algebra},
 FJournal = {Journal f{\"u}r die Reine und Angewandte Mathematik},
 Journal = {J. Reine Angew. Math.},
 ISSN = {0075-4102},
 Volume = {742},
 Pages = {241--261},
 Year = {2018},
 Language = {English},
 DOI = {10.1515/crelle-2015-0091},
 zbMATH = {6930689},
 Zbl = {1416.46057}
}

@Article{deeleyputnamstrung:2024,
 Author = {Deeley, Robin J. and Putnam, Ian F. and Strung, Karen R.},
 Title = {Classifiable {{\(\mathrm{C}^*\)}}-algebras from minimal {{\(\mathbb{Z} \)}}-actions and their orbit-breaking subalgebras},
 FJournal = {Mathematische Annalen},
 Journal = {Math. Ann.},
 ISSN = {0025-5831},
 Volume = {388},
 Number = {1},
 Pages = {703--729},
 Year = {2024},
 Language = {English},
 DOI = {10.1007/s00208-022-02526-1},
 zbMATH = {7796280}
}

@Article{DenkerUrbanski:1991,
 Author = {Denker, Manfred and Urba{\'n}ski, Mariusz},
 Title = {On the existence of conformal measures},
 FJournal = {Transactions of the American Mathematical Society},
 Journal = {Trans. Am. Math. Soc.},
 ISSN = {0002-9947},
 Volume = {328},
 Number = {2},
 Pages = {563--587},
 Year = {1991},
 Language = {English},
 DOI = {10.2307/2001795},
 zbMATH = {25970},
 Zbl = {0745.58031}
}

@book{efhn:2016,
	author = {Tanja Eisner AND B{\'a}lint Farkas AND Markus Haase AND Rainer Nagel},
	number = {272},
	publisher = {Springer},
	series = {Graduate Texts in Mathematics},
	title = {Operator Theoretic Aspects of Ergodic Theory},
	year = {2016}}

@misc{EllGonLinNiu:ClaFinDecRan,
      title={On the classification of simple amenable $C^*$-algebras with finite decomposition rank, II}, 
      author={George A. Elliott and Guihua Gong and Huaxin Lin and Zhuang Niu},
      year={2023},
      eprint={1507.03437},
      archivePrefix={arXiv},
      primaryClass={math.OA},
      url={https://arxiv.org/abs/1507.03437}, 
}

@Article{Exel:1994,
 Author = {Exel, Ruy},
 Title = {Circle actions on {{\(C^*\)}}-algebras, partial automorphisms, and a generalized {Pimsner}-{Voiculescu} exact sequence},
 FJournal = {Journal of Functional Analysis},
 Journal = {J. Funct. Anal.},
 ISSN = {0022-1236},
 Volume = {122},
 Number = {2},
 Pages = {361--401},
 Year = {1994},
 Language = {English},
 DOI = {10.1006/jfan.1994.1073},
 zbMATH = {599806},
 Zbl = {0808.46091}
}

@article{Exel:1998,
 author = {Exel, Ruy},
 title = {Partial actions of groups and actions of inverse semigroups},
 fjournal = {Proceedings of the American Mathematical Society},
 journal = {Proc. Am. Math. Soc.},
 issn = {0002-9939},
 volume = {126},
 number = {12},
 pages = {3481--3494},
 year = {1998},
 language = {English},
 doi = {10.1090/S0002-9939-98-04575-4},
 zbMATH = {1209030},
 Zbl = {0910.46041}
}

@Book{Exel:2017,
 Author = {Exel, Ruy},
 Title = {Partial dynamical systems, {Fell} bundles and applications},
 FSeries = {Mathematical Surveys and Monographs},
 Series = {Math. Surv. Monogr.},
 ISSN = {0076-5376},
 Volume = {224},
 ISBN = {978-1-4704-3785-5; 978-1-4704-4236-1},
 Year = {2017},
 Publisher = {Providence, RI: American Mathematical Society (AMS)},
 Language = {English},
 DOI = {10.1090/surv/224},
 zbMATH = {6801030},
 Zbl = {1405.46003}
}

@Article{Fell:1961,
 Author = {Fell, J. M. G.},
 Title = {The structure of algebras of operator fields},
 FJournal = {Acta Mathematica},
 Journal = {Acta Math.},
 ISSN = {0001-5962},
 Volume = {106},
 Pages = {233--280},
 Year = {1961},
 Language = {English},
 DOI = {10.1007/BF02545788},
 zbMATH = {3164887},
 Zbl = {0101.09301}
}

@misc{ForoughJeongStrung:2025,
      title={Recursive subhomogeneity of orbit-breaking subalgebras of $\mathrm{C}^*$-algebras associated to minimal homeomorphisms twisted by line bundles}, 
      author={Marzieh Forough and Ja A Jeong and Karen R. Strung},
      year={2025},
      eprint={2410.07424},
      archivePrefix={arXiv},
      primaryClass={math.OA},
      url={https://arxiv.org/abs/2410.07424}, 
}

@Article{Geffen2021,
 Author = {Geffen, Shirly},
 Title = {Nuclear dimension of crossed products associated to partial dynamical systems},
 FJournal = {Journal of Functional Analysis},
 Journal = {J. Funct. Anal.},
 ISSN = {0022-1236},
 Volume = {281},
 Number = {2},
 Pages = {31},
 Note = {Id/No 109031},
 Year = {2021},
 Language = {English},
 DOI = {10.1016/j.jfa.2021.109031},
 zbMATH = {7336935},
 Zbl = {1478.46063}
}

@misc{GongLinNiu:ZClass,
      title={Classification of finite simple amenable $\mathcal{Z}$-stable $C^*$-algebras}, 
      author={Guihua Gong and Huaxin Lin and Zhuang Niu},
      year={2015},
      eprint={1501.00135},
      archivePrefix={arXiv},
      primaryClass={math.OA},
      url={https://arxiv.org/abs/1501.00135}, 
}

@inproceedings{Giordano2018,
author="Giordano, Thierry
and Kerr, David
and Phillips, N. Christopher
and Toms, Andrew",
editor="Perera, Francesc",
title="An Introduction to Crossed Products by Minimal Homeomorphisms",
bookTitle="Crossed Products of $C^*$-Algebras, Topological Dynamics, and Classification",
year="2018",
publisher="Springer International Publishing",
address="Cham",
pages="221--274",

isbn="978-3-319-70869-0",
doi="10.1007/978-3-319-70869-0_11",
url="https://doi.org/10.1007/978-3-319-70869-0_11"
}

@Article{Hirshberg2017,
 Author = {Hirshberg, Ilan and Wu, Jianchao},
 Title = {The nuclear dimension of {{\(C^\ast\)}}-algebras associated to homeomorphisms},
 FJournal = {Advances in Mathematics},
 Journal = {Adv. Math.},
 ISSN = {0001-8708},
 Volume = {304},
 Pages = {56--89},
 Year = {2017},
 Language = {English},
 DOI = {10.1016/j.aim.2016.08.022},
 zbMATH = {6642255},
 Zbl = {1376.46042}
}

@Book{Husemoller:1993,
 Author = {Husemoller, Dale H.},
 Title = {Fibre bundles.},
 Edition = {3rd ed.},
 FSeries = {Graduate Texts in Mathematics},
 Series = {Grad. Texts Math.},
 ISSN = {0072-5285},
 Volume = {20},
 ISBN = {0-387-94087-1},
 Year = {1993},
 Publisher = {Berlin: Springer-Verlag},
 Language = {English},
 zbMATH = {496152},
 Zbl = {0794.55001}
}

@Article{Katsura:2004,
 Author = {Katsura, Takeshi},
 Title = {On {{\(C^*\)}}-algebras associated with {{\(C^*\)}}-correspondences},
 FJournal = {Journal of Functional Analysis},
 Journal = {J. Funct. Anal.},
 ISSN = {0022-1236},
 Volume = {217},
 Number = {2},
 Pages = {366--401},
 Year = {2004},
 Language = {English},
 DOI = {10.1016/j.jfa.2004.03.010},
 zbMATH = {2131176},
 Zbl = {1067.46054}
}

@Article{Katsura2007,
 Author = {Katsura, Takeshi},
 Title = {Ideal structure of {{\(C^*\)}}-algebras associated with {{\(C^*\)}}-correspondences},
 FJournal = {Pacific Journal of Mathematics},
 Journal = {Pac. J. Math.},
 ISSN = {1945-5844},
 Volume = {230},
 Number = {1},
 Pages = {107--145},
 Year = {2007},
 Language = {English},
 DOI = {10.2140/pjm.2007.230.107},
 zbMATH = {5366211},
 Zbl = {1152.46048}
}

@unpublished{kirchberg:2022,
	Author = {Kirchberg, Eberhard},
	Note = {Preprint},
URL = {https://ivv5hpp.uni-muenster.de/u/echters/ekneu1.pdf, },
	Title = {The Classification of Purely Infinite {$C^*$}-Algebras Using Kasparov’s
Theory}}

@article {kirchbergphillips:2000,
    AUTHOR = {Kirchberg, Eberhard and Phillips, N. Christopher},
     TITLE = {Embedding of exact {$C^*$}-algebras in the {C}untz algebra
              {$\mathcal{O}_2$}},
   JOURNAL = {J. Reine Angew. Math.},
  FJOURNAL = {Journal f\"ur die Reine und Angewandte Mathematik. [Crelle's
              Journal]},
    VOLUME = {525},
      YEAR = {2000},
     PAGES = {17--53},
      ISSN = {0075-4102,1435-5345},
   MRCLASS = {46L05 (19K56 46L35 46L80)},
  MRNUMBER = {1780426},
MRREVIEWER = {Mikael\ R\o rdam},
       DOI = {10.1515/crll.2000.065},
       URL = {https://doi.org/10.1515/crll.2000.065},
}

@Book{Lance:1995,
 Author = {Lance, E. Christopher},
 Title = {Hilbert {{\(C^*\)}}-modules. {A} toolkit for operator algebraists},
 FSeries = {London Mathematical Society Lecture Note Series},
 Series = {Lond. Math. Soc. Lect. Note Ser.},
 ISSN = {0076-0552},
 Volume = {210},
 ISBN = {0-521-47910-X},
 Year = {1995},
 Publisher = {Cambridge: Univ. Press},
 Language = {English},
 zbMATH = {739278},
 Zbl = {0822.46080}
}

@Article{Lin2010,
 Author = {Lin, Huaxin and Phillips, N. Christopher},
 Title = {Crossed products by minimal homeomorphisms},
 FJournal = {Journal f{\"u}r die Reine und Angewandte Mathematik},
 Journal = {J. Reine Angew. Math.},
 ISSN = {0075-4102},
 Volume = {641},
 Pages = {95--122},
 Year = {2010},
 Language = {English},
 DOI = {10.1515/CRELLE.2010.029},
 zbMATH = {5708786},
 Zbl = {1196.46047}
}

@article{McClanahan:1995,
 author = {McClanahan, Kevin},
 title = {{{\(K\)}}-theory for partial crossed products by discrete groups},
 fjournal = {Journal of Functional Analysis},
 journal = {J. Funct. Anal.},
 issn = {0022-1236},
 volume = {130},
 number = {1},
 pages = {77--117},
 year = {1995},
 language = {English},
 doi = {10.1006/jfan.1995.1064},
 zbMATH = {764040},
 Zbl = {0867.46045}
}

@Article{Nilsen:1996,
 Author = {Nilsen, May},
 Title = {{{\(C^*\)}}-bundles and {{\(C_ 0(X)\)}}-algebras},
 FJournal = {Indiana University Mathematics Journal},
 Journal = {Indiana Univ. Math. J.},
 ISSN = {0022-2518},
 Volume = {45},
 Number = {2},
 Pages = {463--477},
 Year = {1996},
 Language = {English},
 DOI = {10.1512/iumj.1996.45.1086},
 zbMATH = {975912},
 Zbl = {0872.46029}
}

@Misc{Pears1975,
 Author = {Pears, A. R.},
 Title = {Dimension theory of general spaces},
 Year = {1975},
 Language = {English},
 HowPublished = {Cambridge etc.: {Cambridge} {University} {Press}. {XII}, 428 p. {{\textsterling}} 16.50 (1975).},
 zbMATH = {3488230},
 Zbl = {0312.54001}
}

@article {phillips:2000,
    AUTHOR = {Phillips, N. Christopher},
     TITLE = {A classification theorem for nuclear purely infinite simple
              {$C^*$}-algebras},
   JOURNAL = {Doc. Math.},
  FJOURNAL = {Documenta Mathematica},
    VOLUME = {5},
      YEAR = {2000},
     PAGES = {49--114},
      ISSN = {1431-0635,1431-0643},
   MRCLASS = {46L05 (19K56 46L35 46L80)},
  MRNUMBER = {1745197},
MRREVIEWER = {Mikael\ R\o rdam},
}

@Article{phillips:2007,
 Author = {Phillips, N. Christopher},
 Title = {Recursive subhomogeneous algebras},
 FJournal = {Transactions of the American Mathematical Society},
 Journal = {Trans. Am. Math. Soc.},
 ISSN = {0002-9947},
 Volume = {359},
 Number = {10},
 Pages = {4595--4623},
 Year = {2007},
 Language = {English},
 DOI = {10.1090/S0002-9947-07-03850-0},
 zbMATH = {5172034},
 Zbl = {1125.46044}
}

@misc{phillips:2014,
 author = {Phillips, N. Christopher},
 title = {Large subalgebras},
 year = {2014},
 howpublished = {Preprint, {arXiv}:1408.5546 [math.{OA}] (2014)},
 url = {https://arxiv.org/abs/1408.5546},
 arXiv = {arXiv:1408.5546}
}

@InCollection{Pimsner:1997,
 Author = {Pimsner, Michael V.},
 Title = {A class of {{\(C^*\)}}-algebras generalizing both {C}untz-{K}rieger algebras and crossed products by {{\(\mathbb{Z}\)}}},
 BookTitle = {Free probability theory. Papers from a workshop on random matrices and operator algebra free products, Toronto, Canada, Mars 1995},
 ISBN = {0-8218-0675-0},
 Pages = {189--212},
 Year = {1997},
 Publisher = {Providence, RI: American Mathematical Society},
 Language = {English},
 zbMATH = {1003154},
 Zbl = {0871.46028}
}

@Book{rordam:2002,
 Author = {R{\o}rdam, Mikael and St{\o}rmer, Erling},
 Title = {Classification of nuclear {{\(C^*\)}}-algebras. {Entropy} in operator algebras},
 FSeries = {Encyclopaedia of Mathematical Sciences},
 Series = {Encycl. Math. Sci.},
 ISSN = {0938-0396},
 Volume = {126},
 ISBN = {3-540-42305-2},
 Year = {2002},
 Publisher = {Berlin: Springer},
 Language = {English},
 zbMATH = {1683520},
 Zbl = {0985.00012}
}

@Article{Schafhauser:2016,
 Author = {Schafhauser, Christopher},
 Title = {Traces on topological-graph algebras},
 FJournal = {Ergodic Theory and Dynamical Systems},
 Journal = {Ergodic Theory Dyn. Syst.},
 ISSN = {0143-3857},
 Volume = {38},
 Number = {5},
 Pages = {1923--1953},
 Year = {2016},
 Language = {English},
 DOI = {10.1017/etds.2016.114},
 zbMATH = {6908429},
 Zbl = {1396.37013}
}

@article{Spielberg:2007,
 author = {Spielberg, Jack},
 title = {Graph-based models for {Kirchberg} algebras},
 fjournal = {Journal of Operator Theory},
 journal = {J. Oper. Theory},
 issn = {0379-4024},
 volume = {57},
 number = {2},
 pages = {347--374},
 year = {2007},
 language = {English},
 zbMATH = {5198358},
 Zbl = {1164.46028}
}

@Book{Strung:2021,
 Author = {Strung, Karen},
 Editor = {Perera, Francesc},
 Title = {An introduction to {{\(C^*\)}}-algebras and the classification program. {Edited} by {Francesc} {Perera}},
 FSeries = {Advanced Courses in Mathematics -- CRM Barcelona},
 Series = {Adv. Courses in Math. -- CRM Barc.},
 ISSN = {2297-0304},
 ISBN = {978-3-030-47464-5; 978-3-030-47465-2},
 Year = {2021},
 Publisher = {Cham: Birkh{\"a}user/Springer},
 Language = {English},
 DOI = {10.1007/978-3-030-47465-2},
 zbMATH = {7283207},
 Zbl = {1479.46003}
}

@Article{Winter2004,
 Author = {Winter, Wilhelm},
 Title = {Decomposition rank of subhomogeneous {{\(C^*\)}}-algebras},
 FJournal = {Proceedings of the London Mathematical Society. Third Series},
 Journal = {Proc. Lond. Math. Soc. (3)},
 ISSN = {0024-6115},
 Volume = {89},
 Number = {2},
 Pages = {427--456},
 Year = {2004},
 Language = {English},
 DOI = {10.1112/S0024611504014716},
 zbMATH = {2111362},
 Zbl = {1081.46049}
}

@article{winterzacharias:2010,
 author = {Winter, Wilhelm and Zacharias, Joachim},
 title = {The nuclear dimension of {{\(C^{*}\)}}-algebras},
 fjournal = {Advances in Mathematics},
 journal = {Adv. Math.},
 issn = {0001-8708},
 volume = {224},
 number = {2},
 pages = {461--498},
 year = {2010},
 language = {English},
 doi = {10.1016/j.aim.2009.12.005},
 zbMATH = {5706196},
 Zbl = {1201.46056}
}

@article{Wu:2025,
 author = {Wu, Victor},
 title = {{{\(C^\ast\)}}-algebras associated to directed graphs of groups, and models of {Kirchberg} algebras},
 fjournal = {Journal of Functional Analysis},
 journal = {J. Funct. Anal.},
 issn = {0022-1236},
 volume = {288},
 number = {3},
 pages = {39},
 note = {Id/No 110740},
 year = {2025},
 language = {English},
 doi = {10.1016/j.jfa.2024.110740},
 zbMATH = {7962959}
}

@misc{GongLinNiu:ZClass2,
      title={A classification of finite simple amenable Z-stable $C^*$-algebras, II, --$C^*$-algebras with rational generalized tracial rank one}, 
      author={Guihua Gong and Huaxin Lin and Z. Niu},
      year={2021},
      eprint={1909.13382},
      archivePrefix={arXiv},
      primaryClass={math.OA},
      url={https://arxiv.org/abs/1909.13382}, 
}

@article{TWW,
   title={Quasidiagonality of nuclear $C^*$-algebras},
   volume={185},
   ISSN={0003-486X},
   url={http://dx.doi.org/10.4007/annals.2017.185.1.4},
   DOI={10.4007/annals.2017.185.1.4},
   number={1},
   journal={Annals of Mathematics},
   publisher={Annals of Mathematics},
   author={Tikuisis, Aaron and White, Stuart and Winter, Wilhelm},
   year={2017},
   month=jan }

@misc{CarGabeSchafTikWhi2023,
      title={Classifying $\ast$-homomorphisms I: Unital simple nuclear $C^*$-algebras}, 
      author={José R. Carrión and James Gabe and Christopher Schafhauser and Aaron Tikuisis and Stuart White},
      year={2023},
      eprint={2307.06480},
      archivePrefix={arXiv},
      primaryClass={math.OA},
      url={https://arxiv.org/abs/2307.06480}, 
}


\end{document}